\documentclass[11pt,reqno]{amsart}

\usepackage[margin=1in]{geometry}
\usepackage[T1]{fontenc}
\usepackage[utf8]{inputenc}
\usepackage{lmodern}
\usepackage{microtype}
\usepackage{amsmath,amssymb,mathtools,amsthm,xcolor,stmaryrd}
\usepackage{booktabs,array,tabularx,longtable}
\usepackage{float}
\usepackage{enumitem}
\usepackage[
  backend=biber,
  style=numeric,
  sorting=nyt,
  giveninits=true,
  doi=false,
  url=false,
  maxbibnames=99,
  backref=true
]{biblatex}
\usepackage[colorlinks=true,linkcolor=blue,citecolor=blue,urlcolor=blue]{hyperref}
\hypersetup{
  pdftitle={Irrational families of Fano weighted complete intersection fourfolds},
  pdfauthor={Jiaji Cai, Andrew Harder, Sukjoo Lee}
}
\numberwithin{equation}{section}

\theoremstyle{plain}
\newtheorem{thm}{Theorem}[section]
\newtheorem{lem}[thm]{Lemma}
\newtheorem{prop}[thm]{Proposition}
\newtheorem{cor}[thm]{Corollary}

\theoremstyle{definition}
\newtheorem{definition}[thm]{Definition}

\theoremstyle{remark}
\newtheorem{rem}[thm]{Remark}
\newtheorem{question}[thm]{Question}

\newcommand{\PP}{\mathbb P}
\newcommand{\CC}{\mathbb C}
\newcommand{\OO}{\mathcal O}
\newcommand{\ev}{\mathrm{ev}}
\newcommand{\reg}{\mathrm{reg}}
\newcommand{\pr}{\mathrm{pr}}
\newcommand{\eff}{\mathrm{eff}}
\newcommand{\sm}{\mathrm{sm}}
\newcommand{\amb}{\mathrm{amb}}
\newcommand{\inv}{\mathrm{inv}}
\newcommand{\rad}{\mathrm{rad}}
\newcommand{\llaur}{(\!(}
\newcommand{\rlaur}{)\!)}

\DeclareMathOperator{\Spf}{Spf}
\DeclareMathOperator{\rk}{rk}
\DeclareMathOperator{\diag}{diag}
\DeclareMathOperator{\id}{id}
\DeclareMathOperator{\Exp}{Exp}
\DeclareMathOperator{\QDM}{QDM}
\DeclareMathOperator{\End}{End}
\DeclareMathOperator{\Frac}{Frac}
\DeclareMathOperator{\Ker}{ker}
\DeclareMathOperator{\im}{im}
\DeclareMathOperator{\Hom}{Hom}
\DeclareMathOperator{\Mat}{Mat}
\DeclareMathOperator{\Span}{Span}
\DeclareMathOperator{\grad}{grad}
\DeclareMathOperator{\Sym}{Sym}
\DeclareMathOperator{\tr}{tr}
\DeclareMathOperator{\vdim}{vdim}

\allowdisplaybreaks
\begin{document}

\title[Irrational families of smooth Fano weighted complete intersection fourfolds]
{Irrational families of Fano weighted complete intersection fourfolds}

\author{Jiaji Cai, Andrew Harder, Sukjoo Lee}


\begin{abstract}
We study the irrationality of smooth Fano complete intersection fourfolds in weighted projective spaces, using an atom theoretic approach based solely on Gromov--Witten invariants. 
We prove symplectic irrationality, and hence irrationality for every smooth member, in a wide range of families for which irrationality was previously known only for very general members.

\end{abstract}

\maketitle
\tableofcontents

\section{Introduction}
\label{sec:introduction}

\subsection{Main results}

Rationality is a central problem in birational geometry.
Recently, Katzarkov--Kontsevich--Pantev--Yu introduced the theory of \emph{Hodge atoms}, 
which combines Gromov--Witten theory and Hodge theory to construct birational invariants \cite{KKPY}. 
One of its main applications is the irrationality of a very general cubic fourfold \cite[Theorem~6.8]{KKPY}, 
a long-standing conjecture in birational geometry.

These ideas have led to several applications to fourfolds.
Donagi--Pantev use enhanced Hodge atoms and numerical $K$-theory to reprove the irrationality of a very general intersection of three quadrics in $\PP^7$ \cite{DonagiPantevThreeQuadrics}.
Benedetti--Manivel--Perrin compute the small quantum cohomology of Gushel--Mukai fourfolds and prove  irrationality for very general members \cite{BenedettiManivelPerrinGM}.
Benedetti--Fay--Gu\'er\'e--Manivel--Perrin give a criterion based on monodromy and cohomological inequalities, without explicit quantum computations, and apply it to K\"uchle fourfolds of type $(c5)$ \cite{BenedettiEtAlAtomicCriterion}.
More recently, Alexeev--Cavenaghi--Galindo--Georgiev--Katzarkov--Muniz Martins developed localization methods for complete intersections in flag varieties and applied small quantum spectra to Hodge-theoretic rationality obstructions \cite{AlexeevEtAlFlagAtoms}.
See also
\cite{BenedettiEtAlVerra,CavenaghiKatzarkovKontsevichAtomsSymbols,
FayEquivariantVerra,FayHasseObstructions,
BohningBothmerSuaNaiveAtoms,GyengeProductQDM,GyengeSzaboBlowups}
for related developments.

In this paper, we study the irrationality of fourfolds using formal atoms without Hodge-theoretic data, continuing along the lines of \cite{CaiQuarticSymplecticIrrationality,CaiCubicSymplecticIrrationality}.
Although this loses the Hodge-theoretic obstructions available in the theory of Hodge atoms, 
through additional hard work on the Gromov--Witten side, we obtain irrationality results for \emph{all smooth members} of the families under consideration, 
rather than only for general members. 
Such all-member statements are less accessible by existing methods in birational geometry.

In fact, we prove irrationality in the symplectic category.
A symplectic manifold is \emph{symplectically rational} if it is related to projective space by symplectic blow-ups, blow-downs, and deformations.
Otherwise, it is \emph{symplectically irrational}.
For a smooth projective variety with symplectic structure induced by an ample polarization, algebraic rationality implies symplectic rationality by weak factorization \cite{AKMWWeakFactorization}. 

Our main result concerns the 22 families of smooth well-formed Fano complete intersection fourfolds in weighted projective spaces
\cite[Table~1]{PrzyjalkowskiShramovBounds}.
Among these families, $X_2\subset\PP^5\text{ and }X_{2,2}\subset\PP^6$ are always rational, and $X_3\subset\PP^5 \text{ and } X_{2,2,2}\subset\PP^7$ have smooth rational members
\cite{BolognesiRussoStaglianoRationalCubics,
HassettPirutkaTschinkelThreeQuadrics}. Hence they are symplectically rational by definition.

\begin{thm}
  \label{ci:thm:main}
  Every smooth member of the remaining 18 families is irrational. Among these 18 families, all except $X_{2,6}\subset\PP(1^6,3)$ are symplectically irrational.
\end{thm}

See Table~\ref{ci:tab:irrationality} for details.
Very general (stable) irrationality was previously known for some of the
families covered by Theorem~\ref{ci:thm:main};
see \cite{TotaroHypersurfaces,ChatzistamatiouLevineTorsion,
NicaiseOttemTropical,SkauliQuadricCubic} for projective complete intersections, and
\cite{OkadaWeightedHypersurfaces,HassettPirutkaTschinkelDoubleFourfold}
for weighted hypersurfaces.
There are also earlier all-member results in special cases;
for example, every smooth quintic fourfold is birationally superrigid
\cite{DeFernexRigidHypersurfaces,DeFernexRigidHypersurfacesErratum}.

\begin{rem}
    To be precise, we only prove that the 17 families excluding $X_{2,6}\subset\PP(1^6,3)$ are all irrational; the exceptional family $X_{2,6}\subset\PP(1^6,3)$ was already known to be irrational by the work of Pukhlikov \cite{PukhlikovDoubleQuadric}. See Remark \ref{ci:rem:weighted}.
\end{rem}

\begin{rem}
    The approach can be applied to families of threefolds beyond weighted complete intersections, and will be pursued in forthcoming work.
\end{rem} 

For a smooth well-formed Fano weighted complete intersection
$X=X_{d_1,\ldots,d_c}\subset\PP(a_0,\ldots,a_{m+c})$, which is not an intersection with a linear cone, 
consider the numerical quantity
\[
  \sigma(X):=\dim X-\frac{2\iota}{d_{\max}},
\]
where $\iota=\sum_i a_i-\sum_j d_j$ is the Fano index and $d_{\max}=\max_j d_j$.
For complete intersections in ordinary projective space, this equals
the upper Serre dimension of the residual category in
$D^b\mathrm{Coh}(X)$;
see \cite[Theorem~1.7]{KuznetsovPerrySerre}.
See also \cite{CarazolliEtAlQuantumRationality}
for a related quantum-connection perspective.
\begin{question}[Katzarkov--Kontsevich]
    Does the numerical condition
    \[\sigma(X)>\dim X-2\quad \textrm{ or equivalently }\quad \iota<d_{\max}\]
    imply that very general members are irrational? Does it imply that all smooth members are irrational?
\end{question} 
Theorem \ref{ci:thm:main} fully confirms this speculation in dimension 4:

\begin{cor}
\label{ci:cor:spectral-dimension}
Among the 22 families of smooth well-formed Fano weighted complete intersection fourfolds:
\begin{itemize}
    \item if $\sigma(X)>2$, then all smooth members are irrational;
    \item if $\sigma(X)=2$, then some members are rational and some are irrational;
    \item if $\sigma(X)=0,1$, then all smooth members are rational.
\end{itemize}
\end{cor}

\subsection{Summary}

We briefly describe the construction of formal atoms and the strategy
of the proof of Theorem~\ref{ci:thm:main}.
Set
\[
  H_X=H^{\ev}(X,\CC),\qquad
  k_X=\overline{\Frac(R_X)},\qquad
  V_X=H_X\otimes_\CC k_X,
\]
where $R_X$ is the coefficient ring for the big quantum product.
The $z$-component of the quantum connection is
\[
  \nabla_z
  =
  \partial_z-\frac{C_E}{z^2}+\frac{\mu_X}{z},
\]
where $C_E$ denotes quantum multiplication by the Euler element.
After extending scalars to $k_X((z))$,
Proposition~\ref{prop:euler-spectral-block} gives an orthogonal
formal decomposition
\[
  (V_X((z)),\nabla,P_X)
  \simeq
  \bigoplus_\Lambda^\perp
  (\mathcal M_\Lambda,\nabla_\Lambda,P_X|_{\mathcal M_\Lambda}).
\]
Each summand $\mathcal M_\Lambda$ is called the \emph{formal atom}
$\mathrm{At}(X)_\Lambda$ (or simply an atom).

To each formal atom we associate its multiplicity $m_\Lambda$,
Jordan partition $\mathcal J_\Lambda$ of $C_E|_{V_\Lambda}$, and
formal type of its $z$-connection.
By the Hukuhara--Levelt--Turrittin (HLT) theorem, after finite ramification
the latter is described by exponential factors
$\varphi_{\Lambda,a}$ and regular-singular factors $R_{\Lambda,a}$.
We refer to the data
\[
  \left(
    m_\Lambda,\mathcal J_\Lambda,
    \{(\varphi_{\Lambda,a},[R_{\Lambda,a}])\}_a
  \right)
\] as \emph{formal atomic invariants}.
For an atom of unramified exponential type,
\[
  \mathcal M_\Lambda
  \simeq
  \mathcal E^{-\Lambda/z}\otimes R_\Lambda,
\]
where $  R_\Lambda$ is regular singular and 
\[
  \mathcal E^{-\Lambda/z}
  =
  \left(k_X((z)),\partial_z-\frac{\Lambda}{z^2}\right)
\]
is the rank-one exponential connection with horizontal solution
$e^{-\Lambda/z}$. In this case, we also use the coarser triple
\[
  (m_\Lambda,\mathcal J_\Lambda,\Exp_\Lambda),
\]
where $\Exp_\Lambda=\Exp(R_\Lambda)\subset\CC/\mathbb Z$. 
When studying irrationality, we can argue by contradiction and assume the variety in question is rational, and then unramifiedness comes for free (at least in lower dimensions).

The basic source of irrationality obstructions is Iritani's quantum
blow-up formula
\cite[Theorems~1.1 and~5.18]{IritaniBlowups}
and its symplectic extension
\cite[Theorem~3.5]{CaiQuarticSymplecticIrrationality}.
Let $\widetilde X$ be the blow-up of $X$ along a smooth center
$Z$ of complex codimension $r\geq2$.
In terms of formal atoms, this says schematically that
\[
  \mathrm{At}(\widetilde X)
  \simeq
  \mathrm{At}(X)
  \sqcup
  \bigsqcup_{j=0}^{r-2}\mathrm{At}(Z).
\]

Consequently, if a fourfold is symplectically rational, its formal
atoms must be built from those of $\PP^4$ and of blow-up centers of
complex dimension at most two, namely points, curves, and
surfaces (symplectic manifolds of complex dimension 2). 
Their formal atoms are described in
Section~\ref{sec:obstruction}; see also
\cite{GyengeSzaboBlowups} for related surface calculations.
Thus, to obstruct rationality, it is enough to find a formal atom
whose atomic invariants cannot occur among these lower-dimensional
contributions.

The main difficulty is that the big quantum connection is usually
hard to compute directly. By contrast, the small quantum connection
of a weighted complete intersection is accessible through the mirror
theorem \cite{PrzyjalkowskiWeightedQH}.
We therefore begin with the small or other restricted quantum
connections and study which parts of their atomic data are preserved
under bulk deformation.

The cohomology of a smooth weighted complete intersection decomposes into an ambient part, obtained by restriction from the ambient space, and a primitive part in the middle degree.
When deforming in the ambient directions, atomic invariants are preserved by the flow of the quantum powers of the Euler vector field; 
see \cite[Lemmas~5.4--5.5]{CaiQuarticSymplecticIrrationality} and Theorem~\ref{thm:formal-subscheme-persistence}.
For the primitive directions, the main idea is to use monodromy reduction \cite{TianGangQH,CollinoJinzenjiHypersfQH,HuBigQHCI} to control the splitting of atoms.

For the 17 families for which Theorem \ref{ci:thm:main} asserts symplectic irrationality, 
the Zariski closure of the monodromy representation is the orthogonal representation on primitive cohomology;
see Proposition~\ref{ci:prop:weighted-monodromy}. 
Thus the dependence of the Gromov--Witten potential on the primitive variables reduces to a single variable.
This reduction produces a primitive part of the formal atom
corresponding to a distinguished eigenvalue $\Lambda$: 
\[
  \mathcal M_\Lambda
  \simeq
  \mathcal E^{-\Lambda/z}\otimes
  \bigl(\mathbf1^{\oplus N}\oplus\mathcal R_\Lambda\bigr),
\]
where $\mathbf1^{\oplus N}$ is the trivial differential module of rank $N=\dim H^4_{\pr}(X)$; see Proposition~\ref{ci:prop:atom-constraints}.
The presence of the rank-$N$ trivial summand severely restricts the possible surface atoms occurring in a hypothetical factorization sequence.

The remaining argument depends on the Fano index of the variety in question.
In index three, the size and Jordan type of the $\Lambda$-atom give a contradiction.
In index two, the $\Lambda$-atom does not split, and then Proposition~\ref{lem:non-split-unramified} allows us to compare its regular exponents with those computed on the small locus. 
The resulting non-integral exponents give a contradiction. 
In index one, the size and Jordan type of the $\Lambda$-atom force a specific type of splitting, which we exclude using Frobenius identities.

As a separate point of interest, in the index-two case we determine
the multiplicities and Jordan partitions of $C_E$ directly (without arguing by contradiction), and obtain the polynomial relation
\eqref{ci:eq:big-index-two-relation}; see
Proposition~\ref{ci:prop:exact-index-two}.

The paper is organized as follows.
In Section~\ref{sec:spectral-data}, we introduce formal atoms, and study how atomic invariants change under specialization and deformation.
In Section~\ref{sec:obstruction}, we study the formal atoms of possible low-dimensional blow-up centers.
In Section~\ref{sec:ci-fourfolds}, we turn to our weighted complete
intersections, and prove Theorem~\ref{ci:thm:main}.

\section*{Acknowledgements}
\noindent 
J.C. would like to thank Mark McLean, Juan Mu\~noz-Ech\'aniz, Paul Seidel, Zihong Chen, Jae Hee Lee for helpful conversations.
A.H. was supported by the Simons Foundation's Travel Support for Mathematicians program during the completion of this work.
S.L. was supported by the Institute for Basic Science (IBS-R003-D1). 

\medskip
\noindent\textbf{AI disclosure.}
ChatGPT was used in the preparation of this manuscript for drafting,
language editing, and consistency checks.
The authors take full responsibility for the mathematical content
of this paper.

\section{Formal atoms}
\label{sec:spectral-data}

\subsection{Quantum \texorpdfstring{$D$}{D}-modules}

We recall the notion of the big quantum $D$-module. We work throughout with even cohomology
\[
  H_X:=H^{\ev}(X,\CC)=\bigoplus_j H^{2j}(X,\CC).
\]
Let $(X,\omega)$ be a closed symplectic manifold and let
$C^{\eff}(X,\omega)\subset H_2(X,\mathbb Z)/\mathrm{tors}$ be the
monoid generated by curve classes occurring in genus-zero
Gromov--Witten theory.  For $d\in C^{\eff}(X,\omega)$ write $Q^d$ for
the corresponding Novikov monomial.  Let $\CC\llbracket Q\rrbracket $ be the completion
with respect to symplectic energy: a formal sum $\sum_d a_dQ^d$ is
allowed when, for every $c\in\mathbb R$, only finitely many terms with
$\omega\cdot d<c$ have nonzero coefficient.  Multiplication is the
continuous Cauchy extension of
$Q^{d_1}Q^{d_2}=Q^{d_1+d_2}$. The grading of $\CC\llbracket Q\rrbracket $ is given by
$\deg Q^d=2c_1(X)\cdot d$.
Choose a homogeneous basis
$\phi_0=1,\phi_1,\ldots,\phi_{m-1}$ of $H_X$ and write
$\tau=\sum_i\tau^i\phi_i$, where
$\deg\tau^i=2-\deg\phi_i$.  Set
\[
  R_X=\CC\llbracket Q\rrbracket \llbracket \tau^0,\ldots,\tau^{m-1}\rrbracket .
\]
Thus one first completes in the Novikov variables with respect to energy
and then completes formally in the bulk variables.

\begin{definition}
  The (universal) big quantum product is the $R_X$-bilinear product
  $\star$ on $H_X\otimes R_X$ determined on the basis elements by
  \[
    (\phi_i\star\phi_j,\phi_k)_X
    =
    \sum_{d,n}
    \left\langle
    \phi_i,\phi_j,\phi_k,
    \underbrace{\tau,\ldots,\tau}_{n}
    \right\rangle^X_{0,n+3,d}
    \frac{Q^d}{n!}.
  \]
\end{definition}

Let $\mathcal B_X:=\Spf(R_X)\longrightarrow\Spf(\CC\llbracket Q\rrbracket )$ be the
formal bulk germ over the Novikov base.  Its relative tangent
sheaf is identified with the even cohomology by
\[
  \Theta_{\mathcal B_X/\CC\llbracket Q\rrbracket }
  =
  \bigoplus_iR_X\frac{\partial}{\partial\tau^i}
  \longrightarrow
  H_X\otimes R_X,
  \qquad
  \frac{\partial}{\partial\tau^i}\longmapsto\phi_i.
\]
Transporting the big quantum product across this identification defines
a multiplication $\circ$ on relative vector fields:
\[
  \frac{\partial}{\partial\tau^i}\circ
  \frac{\partial}{\partial\tau^j}
  \longleftrightarrow
  \phi_i\star\phi_j.
\]
The unit is $e=\partial_{\tau^0}$.  The Euler vector field $E$ is the
relative vector field corresponding to the Euler element $E_X$ defined
below.  The WDVV, unit, and grading identities make
$(\mathcal B_X,\circ,e,E)$ a formal $F$-manifold over $\CC\llbracket Q\rrbracket $; see
\cite[Chapter~I, Sections~2 and~5]{ManinFrobenius}.
All tangent spaces, formal subschemes, and formal flows below are
understood relative to the Novikov base.

If $A$ is a complete topological $\CC$-algebra, every continuous
$\CC$-algebra homomorphism $R_X\to A$ gives an $A$-valued specialization
of the universal product.  The small quantum product is the restriction
$\star_{\sm}=\star|_{\tau=0}$ over $\CC\llbracket Q\rrbracket $.

Suppose $\dim_{\CC}X=n$.  The grading operator and Euler element
are
\[
  \mu_X(\phi_i)
  =
  \left(\frac{\deg\phi_i}{2}-\frac n2\right)\phi_i,
  \qquad
  E_X
  =
  c_1(X)+\sum_i\left(1-\frac{\deg\phi_i}{2}\right)\tau^i\phi_i.
\]
We write $C_E=E_X\star(-)$.  The $z$-direction of the quantum
connection is
\begin{equation}
  \nabla_z
  =
  \partial_z-\frac{C_E}{z^2}+\frac{\mu_X}{z}.
  \label{eq:quantum-z-connection}
\end{equation}
For the relative bulk directions, set
$\nabla_{\tau^i}=\partial_{\tau^i}+z^{-1}(\phi_i\star)$.  The full
quantum connection also has logarithmic components along the Novikov
base: for $\xi\in H^2(X)$,
$\nabla_{\xi Q\partial_Q}=\xi Q\partial_Q+z^{-1}(\xi\star)$, where
$(\xi Q\partial_Q)(Q^d)=(\xi\cdot d)Q^d$.

\begin{prop}
  \label{prop:quantum-flatness}
  These operators define a flat meromorphic connection.
\end{prop}

\begin{proof}
  The commutativity of the bulk connections is the WDVV equation, the
  Novikov components are compatible by the divisor equation, and
  compatibility with \eqref{eq:quantum-z-connection} is the usual grading
  identity for the genus-zero potential; see
  \cite[Sections~2--3]{KontsevichManinGW} and
  \cite[Chapter~II]{ManinFrobenius}.
\end{proof}

We denote the $R_X[z,z^{-1}]$-module
$H_X\otimes_\CC R_X[z,z^{-1}]$ equipped with this connection by
$\QDM(X)$.  
The flat pairing $P_X$ is the Poincar\'e pairing of a section at $-z$ with another at $z$.

\subsection{Formal spectral decomposition}

Set $k_X=\overline{\Frac(R_X)}$
and $V_X=H_X\otimes_\CC k_X$. 
The operator $C_E$ is now a $k_X$-linear endomorphism of
$V_X$.  For an eigenvalue $\Lambda\in k_X$, write the generalized eigenspace as
\[
  V_\Lambda=\Ker\bigl((C_E-\Lambda\id)^m\bigr).
\]
Here $m=\dim H_X$. Since $k_X$ is algebraically closed,
$V_X=\bigoplus_\Lambda V_\Lambda$. 

\begin{definition}
  \label{def:euler-spectral-data}
  For an eigenvalue $\Lambda$ of $C_E$ over $k_X$, define
  \[
    m_\Lambda=\dim_{k_X}V_\Lambda,
    \qquad
    \mathcal J_\Lambda
    =
    \text{the Jordan partition of }C_E|_{V_\Lambda}.
  \]
\end{definition}

For example, $\mathcal J_\Lambda=(3,1)$ means
$C_E|_{V_\Lambda}\sim J_3(\Lambda)\oplus J_1(\Lambda)$ and
$m_\Lambda=4$. 

We extend scalars from the Laurent polynomial ring $R_X[z,z^{-1}]$
to the formal Laurent series field $k_X((z))$. This gives the differential
module with pairing $(V_X((z)),\nabla,P_X)$, whose $z$-component is
\eqref{eq:quantum-z-connection}.
The generalized eigenspace decomposition can be lifted to the full meromorphic connection, compatibly with $P_X$.
\begin{prop}
  \label{prop:euler-spectral-block}
  There is an orthogonal formal decomposition
  \begin{equation}
    (V_X((z)),\nabla,P_X)
    \simeq
    \bigoplus_\Lambda^\perp
    (\mathcal M_\Lambda,\nabla_\Lambda,P_X|_{\mathcal M_\Lambda})
  \end{equation}
  where $\mathcal{M}_\Lambda=G(V_\Lambda((z)))$
  for some gauge $G\in I+z\End_{k_X}(V_X)[[z]]$, the
  transformed connection $\nabla_\Lambda$ preserves every $V_\Lambda((z))$,
  and the restricted pairing
  $P_X|_{\mathcal{M}_\Lambda}$ is nondegenerate and flat.
\end{prop}

\begin{proof}
  Since $C_E$ is self-adjoint, its generalized eigenspaces
  $V_\Lambda$ are mutually orthogonal, and the restricted  Poincar\'e pairing is nondegenerate. However, $\mu_X$ need not
  preserve them, so this is not yet a decomposition of the connection.

  Put $A_\Lambda=C_E|_{V_\Lambda}$. For $\Lambda\ne\Gamma$, the
  Sylvester operator
  \[
    \mathcal S_{\Lambda\Gamma}:
    \Hom_{k_X}(V_\Gamma,V_\Lambda)
    \longrightarrow
    \Hom_{k_X}(V_\Gamma,V_\Lambda),
    \qquad
    X\longmapsto A_\Lambda X-XA_\Gamma
  \]
  is invertible because its semisimple part is $(\Lambda-\Gamma)I$. We use these operators to construct a gauge
  \[
    G(z)=I+G_1z+G_2z^2+\cdots \in I+z\End_{k_X}(V_X)[[z]]
  \]
  making the $z$-connection block diagonal. At the first step,
  \[
    G^{-1}\nabla_zG
    =
    \partial_z-\frac{C_E}{z^2}
    +\frac{\mu_X+[G_1,C_E]}{z}+O(1),
  \]
  so the off-diagonal coefficient is removed by solving
  \[
    \mathcal S_{\Lambda\Gamma}
    \bigl((G_1)_{\Lambda\Gamma}\bigr)
    =(\mu_X)_{\Lambda\Gamma}.
  \]
  At every subsequent order, the unknown off-diagonal coefficient
  enters through the same Sylvester operator, while the remaining
  terms are determined by the previously chosen coefficients.
  Invertibility therefore gives the required gauge recursively.

  Choose a $k_X$-basis adapted to $\bigoplus_\Lambda V_\Lambda$
  and write
  \[
    G^{-1}\nabla_zG
    =\partial_z+\bigoplus_\Lambda\Omega_\Lambda(z),
    \qquad
    \Omega_\Lambda(z)=-A_\Lambda z^{-2}+O(z^{-1}).
  \]
  For a bulk or logarithmic Novikov derivation $\delta$, write
  $G^{-1}\nabla_\delta G=\delta+B_\delta(z)$ in the same basis,
  including derivatives of the basis in $B_\delta$.
  The off-diagonal part of the flatness equation gives
  \[
    \partial_zB_{\delta,\Lambda\Gamma}
    +\Omega_\Lambda B_{\delta,\Lambda\Gamma}
    -B_{\delta,\Lambda\Gamma}\Omega_\Gamma=0.
  \]
  If $B_nz^n$ were its lowest nonzero Laurent term, the coefficient
  of $z^{n-2}$ would give $A_\Lambda B_n-B_nA_\Gamma=0$, 
  contradicting invertibility of $\mathcal S_{\Lambda\Gamma}$.
  Thus every connection component preserves the blocks, and
  $\mathcal M_\Lambda=G(V_\Lambda((z)))$ are differential summands.

  It remains to prove their orthogonality. Let
  $P_{\Lambda\Gamma}(z)$ be the corresponding block of the pairing
  pulled back by $G$, so that
  \[
    P_X(Ga,Gb)
    =
    a(-z)^{\mathsf t}P_{\Lambda\Gamma}(z)b(z),
    \qquad
    a\in V_\Lambda((z)),\quad b\in V_\Gamma((z)).
  \]
  Flatness of the pairing in the $z$-direction gives
  \[
    \partial_zP_{\Lambda\Gamma}
    =
    -\Omega_\Lambda(-z)^{\mathsf t}P_{\Lambda\Gamma}
    +P_{\Lambda\Gamma}\Omega_\Gamma(z).
  \]
  For $\Lambda\ne\Gamma$, a lowest nonzero term $P_nz^n$
  would therefore satisfy $ A_\Lambda^{\mathsf t}P_n-P_nA_\Gamma=0$. Since $A_\Lambda^{\mathsf t}$ and $A_\Gamma$ have disjoint
  spectra, the same Sylvester argument forces $P_n=0$.
  Hence $P_{\Lambda\Gamma}=0$ whenever $\Lambda\ne\Gamma$.

  Finally, $G=I+O(z)$ implies that the constant term of
  $P_{\Lambda\Lambda}(z)$ is the nondegenerate Poincar\'e pairing
  on $V_\Lambda$. Each restricted pairing is therefore
  nondegenerate. It is flat in every direction because the full
  connection preserves $\mathcal M_\Lambda$ and $P_X$ is flat.
\end{proof}

\begin{rem}
  After quotienting the redundant divisor directions by the divisor
  equation, the full quantum connection can be viewed as a maximal formal
  logarithmic $F$-bundle.  In that formulation,
  Proposition~\ref{prop:euler-spectral-block}
is compatible with the formal decomposition in
  \cite[Theorem~3.32]{HYYZZ}; compare also
  \cite[Theorem~4.1]{KKPY}.
\end{rem}

\begin{definition}
    Each summand of the formal decomposition in Proposition \ref{prop:euler-spectral-block} is called a \emph{formal atom} of $X$. We denote it by $\mathrm{At}(X)_\Lambda$
  and write $\mathrm{At}(X)$ for the collection of formal atoms.
\end{definition}

\subsection{Exponent classes}
We recall the Hukuhara--Levelt--Turrittin theorem; see
\cite{HukuharaSingularII,TurrittinIrregular,LeveltJordan}.

\begin{thm}
  \label{thm:levelt-turrittin}
  Let $(\mathcal M,\nabla_z)$ be a finite-dimensional differential module
  over $k_X((z))$ with a $z$-connection $\nabla_z$. After a suitable finite ramification $z=w^m$, there
  is a decomposition
  \begin{equation}
    \mathcal M\otimes_{k_X((z))}k_X((w))
    \simeq
    \bigoplus_\varphi\mathcal E^\varphi\otimes R_\varphi,
  \end{equation}
  where
  \begin{enumerate}
    \item $\varphi\in w^{-1}k_X[w^{-1}]$ and
    $\mathcal E^\varphi=(k_X((w)),d-d\varphi)$;
    \item $R_\varphi$ is regular singular.
  \end{enumerate}
\end{thm}

Fix a minimal ramification $z=w^m$.
The distinct exponential factors $\varphi$ and the isomorphism classes of their
regular factors $[R_\varphi]$ are unique, up to permutation.
Here regular singularity means that $R_\varphi$ admits a full-rank
$k_X[[w]]$-lattice stable under $\nabla_{w\partial_w}$.
A regular-singular module $R$ admits a logarithmic normal form
after a meromorphic change of frame:
\[
  R\simeq\left(k_X((w))^r,d-A\frac{dw}{w}\right),
  \qquad A\in\End_{k_X}(k_X^r).
\]
For the eigenvalues $a_1,\ldots,a_r$ of $A$,  define the
\emph{multiset of exponent classes}
\begin{equation}
  \label{eq:regular-exponents}
  \Exp(R)=\{[a_1],\ldots,[a_r]:[a_i]\in k_X/\mathbb Z\}.
\end{equation}
Changing the logarithmic lattice changes the residue eigenvalues only
by integers, so this multiset is intrinsic; see
\cite[Proposition~2.2.15]{KedlayaGoodFormalI}.
Each HLT summand admits a basis of formal
$w$-horizontal solutions whose scalar components have the form
\[
  e^{\varphi(w)}w^\rho
  \sum_{j=0}^{\ell}f_j(w)(\log w)^j,
  \qquad f_j(w)\in k_X((w)),\quad \rho\in k_X,
\]
where $[\rho]\in\Exp(R_\varphi)$.

\begin{definition}
    A $z$-connection is of \emph{unramified exponential type} if it admits a decomposition
\[
  \mathcal M\simeq
  \bigoplus_\Lambda\mathcal E^{-\Lambda/z}\otimes R_\Lambda,
  \qquad \Lambda\in k_X,
  \qquad R_\Lambda\text{ regular singular in }z.
\]
\end{definition}

Suppose now that $k_X$ carries finitely many commuting derivations
$\delta$, extended by $\delta(z)=0$. An integrable extension
of $\nabla_z$ consists of connection operators $\nabla_\delta$
satisfying the Leibniz rule and
\[
  [\nabla_\delta,\nabla_z]=0,
  \qquad
  [\nabla_\delta,\nabla_{\delta'}]=0.
\]
Extend these derivations by $\delta(w)=0$. The HLT summands are
preserved by all connection components, and the regular factors inherit
integrable connections after the exponential twists are taken in all
directions; see \cite[Theorem~2.3.3]{KedlayaGoodFormalI}.
If
\[
  \bigcap_\delta\Ker(\delta:k_X\to k_X)=\CC,
\]
then $\Exp(R_\varphi)\subset\CC/\mathbb Z$ by \cite[Proposition~2.2.12]
{KedlayaGoodFormalI}. In this case the representatives $\rho$ in the solutions above may be taken in $\CC$.

We apply this discussion to the quantum connection $\nabla$. 
Integrability follows from
Proposition~\ref{prop:quantum-flatness}, and its coefficient derivations satisfy
\begin{equation}
  \label{eq:qdm-joint-constants}
  \bigcap_i\Ker(\partial_{\tau^i}:k_X\to k_X)
  \cap\bigcap_{\xi\in H^2(X,\CC)}
  \Ker(\xi Q\partial_Q:k_X\to k_X)=\CC.
\end{equation}
Applying Theorem~\ref{thm:levelt-turrittin} to $\mathcal M_\Lambda$, we obtain
\begin{equation}
  \mathcal M_\Lambda\otimes_{k_X((z))}k_X((w))
  \simeq
  \bigoplus_a\mathcal E^{\varphi_{\Lambda,a}}\otimes R_{\Lambda,a}.
\end{equation}
after minimal ramification $z=w^m$,
and
\[
  \Exp_{\Lambda,a}:=\Exp(R_{\Lambda,a})\subset\CC/\mathbb Z,
\]
with exponents taken in the coordinate $w$.

To the formal atom $\mathrm{At}(X)_\Lambda$, we assign the following triple, which we call \emph{formal atomic invariants}
\[
  \boxed{
    \left(
      m_\Lambda,\mathcal J_\Lambda,
      \left\{(\varphi_{\Lambda,a},[R_{\Lambda,a}])\right\}_a
    \right).
  }
\]
In applications we also use the coarser data
\[
  \boxed{
    \left(
      m_\Lambda,\mathcal J_\Lambda,
      \left\{(\varphi_{\Lambda,a},\Exp_{\Lambda,a})\right\}_a
    \right).
  }
\]

In particular, when $\mathcal{M}_\Lambda$ is of unramified exponential type, i.e $m=1$, the only exponential factor is $-\Lambda/z$, and
\[
  \mathcal M_\Lambda\simeq
  \mathcal E^{-\Lambda/z}\otimes R_\Lambda,
  \qquad R_\Lambda\text{ regular singular}.
\]
We then omit the fixed exponential factor from the notation and write
$\Exp_\Lambda=\Exp(R_\Lambda)$.

\begin{rem}
    The exponent multiset does not record the size of the logarithmic Jordan blocks. For
$\alpha\in\CC$, put
\[
  \mathcal K_\alpha
  =\left(k_X((w)),d-\alpha\frac{dw}{w}\right).
\]
The module $\mathcal K_\alpha^{\oplus2}$ and a nontrivial
self-extension of $\mathcal K_\alpha$ both have exponent multiset
$\{[\alpha],[\alpha]\}$, but only the latter has a nontrivial
unipotent part of formal monodromy. By contrast, a rank-two regular
factor with exponent classes $[\alpha]\ne[\beta]$ splits into the
two corresponding rank-one factors.
\end{rem}

\begin{rem}
  The small quantum connection of a monotone symplectic
  manifold is of unramified exponential type
  \cite{ChenExponentialType,PomerleanoSeidel}.
  For our purposes, if we assume by contradiction that $X$ is rational, then by Proposition~\ref{prop:spectral-blowup} the atoms of $X$ come from atoms of projective space, points, curves, and surfaces, and thus are of unramified exponential type.
\end{rem}

\subsection{Specialization and deformation}
In practice, computing formal atoms for $\dim_\CC X\geq3$ is difficult
because the big quantum connection is hard to determine. We therefore
restrict to a simpler base, such as $\tau=0$ for the small quantum
connection, and ask which invariants of the restricted connection
extend to the big quantum connection.

For $k\geq0$, let
\[
  E^0=e,
  \qquad
  E^{k+1}=E\circ E^k,
\]
where $\circ$ is the formal $F$-manifold multiplication.

\begin{lem}
  \label{lem:jordan-flow}
  Let $\varphi_k^s$ be the formal flow of $E^k$. Distinct
  eigenvalues of $C_E$ remain distinct along this flow, and each
  corresponding generalized eigenspace retains its algebraic multiplicity
  and Jordan partition.
\end{lem}

\begin{proof}
  The standard $F$-manifold identities give
  \begin{equation}
    \mathcal L_{E^k}C_E=C_E^k;
    \label{eq:euler-flow}
  \end{equation}
  see \cite[I.2 and I.5]{ManinFrobenius}. Thus
  $A(s):=(\varphi_k^s)^*C_E$ satisfies
  \begin{equation}
    \frac{dA}{ds}=A^k,
    \qquad A(0)=C_E.
    \label{eq:euler-flow-ode}
  \end{equation}
  Let $g_{k,s}(x)$ solve $\partial_sg=g^k$ with $g_{k,0}(x)=x$.
  Explicitly, $g_{0,s}(x)=x+s$, $g_{1,s}(x)=e^s x$, and, for $k\geq2$,
  \[
    g_{k,s}(x)
    =
    x\left(1-(k-1)sx^{k-1}\right)^{-1/(k-1)}.
  \]
  Uniqueness gives $A(s)=g_{k,s}(A(0))$, and
  $\partial_xg_{k,s}(x)=1+O(s)$ is a unit.

  On a Jordan block $\lambda I+N$, Taylor expansion gives
  \[
    g_{k,s}(\lambda I+N)-g_{k,s}(\lambda)I=N\,U(N),
    \qquad U(0)=\partial_xg_{k,s}(\lambda).
  \]
  Thus $U(N)$ is invertible and commutes with $N$, so
  \[
    \Ker\left(g_{k,s}(\lambda I+N)-g_{k,s}(\lambda)I\right)^\ell
    =\Ker N^\ell\qquad(\ell\geq1).
  \]
  These kernels determine the Jordan partition and hence the
  multiplicity. Finally,
  \[
    g_{k,s}(\lambda)-g_{k,s}(\mu)
    =(\lambda-\mu)(1+O(s)),
  \]
  so distinct eigenvalues remain distinct.
\end{proof}

\begin{rem}
  The same Euler-flow mechanism is used in
  \cite[Lemmas~5.4--5.5]{CaiQuarticSymplecticIrrationality}.
  Recently, the first author was informed that similar arguments had been carried out in \cite{DavidHertlingEulerFlow} in the context of complex analytic $F$-manifolds.
\end{rem}

Let $S=\Spf(R_S)\subset\mathcal B_X$ be a smooth formal subscheme
  over the fixed Novikov base. After localization along $S$, choose
  normal bulk coordinates $v=(v_1,\ldots,v_c)$ such that the normal
  formal neighborhood is $\Spf R_S\llbracket v\rrbracket $. Write $u$ for local coordinates along $S$. 
  Fix an HLT solution basis and its ramification for the restricted
  $z$-connection. 

\begin{lem}
  \label{prop:formal-slice-solution-extension}
  Every solution $Y_S$ of $\nabla_z|_S\,Y_S=0$ with
  fractional exponent class $[\rho]$ extends to a
  solution $Y$ with $Y|_{v=0}=Y_S$ and the same exponent class.
\end{lem}

\begin{proof}
  Write
  \[
    \nabla_{v_a}=\partial_{v_a}+\frac{P_a(u,v)}z.
  \]
  Flatness gives a unique compatible formal solution
  \[
    \partial_{v_a}M=-\frac{P_a}{z}M,
    \qquad
    M|_{v=0}=I.
  \]
  Expanding $P_a=\sum_\eta P_{a,\eta}v^\eta$ and
  $M=\sum_\nu M_\nu v^\nu$ gives the recursion
  \[
    (\nu_a+1)M_{\nu+e_a}
    =
    -z^{-1}\sum_{\eta+\kappa=\nu}P_{a,\eta}M_\kappa.
  \]
  Since the $P_{a,\eta}$ are independent of $z$, every $M_\nu$ is
  a polynomial in $z^{-1}$. The same holds coefficientwise for $M^{-1}$.
  Moreover,
  \[
    M^{-1}\nabla_{v_a}M=\partial_{v_a},
    \qquad M^{-1}\nabla_zM=\nabla_z|_S,
  \]
  the second identity following from flatness and evaluation at $v=0$.
  Hence $Y=MY_S$ is the required extension. Both $M$ and $M^{-1}$
  change powers of $z$ only by integers, proving preservation of the
  exponent classes, linear independence, and multiplicity. After $z=w^m$, the
  same recursion involves only integral powers of $w$.
\end{proof}

Let $S=\Spf(R_S)\subset\mathcal B_X$ be a smooth integral formal
subscheme of relative codimension $r$ over $\CC\llbracket Q\rrbracket $. Let $K_S=\Frac(R_S)$ and $k_S=\overline{K_S}$, and write
\[
  TS=\Theta_{S/\CC\llbracket Q\rrbracket}\otimes_{R_S}K_S
  \subset H_X\otimes_\CC K_S.
\]
For an eigenvalue $\lambda$ of $C_E|_S$ over $k_S$, define
$(m_\lambda,\mathcal J_\lambda)$ as in
Definition~\ref{def:euler-spectral-data}.
Write $\Exp^{\mathrm{cw}}$ for the fractional-exponent classes of a chosen solution basis on $S$, 
transported coefficientwise along the normal variables as in Lemma \ref{prop:formal-slice-solution-extension}.
A priori, this can be different from the HLT exponents $\Exp$ over $k_X((z))$.

Suppose that quantum powers of the Euler vector field $E^{k_1},\ldots,E^{k_r}$ satisfy
\begin{equation}
  TS+\Span_{K_S}\{E^{k_1}|_S,\ldots,E^{k_r}|_S\}
  =H_X\otimes_\CC K_S.
  \label{eq:formal-subscheme-spanning}
\end{equation}

\begin{thm}
  \label{thm:formal-subscheme-persistence}
  Assume \eqref{eq:formal-subscheme-spanning}. 
  After an extension of coefficient fields, every eigenvalue $\lambda$ of
  $C_E|_S$ determines an eigenvalue $\Lambda$ of $C_E$ over $k_X$, such that
  \[
    m_\Lambda=m_\lambda,
    \qquad
    \mathcal J_\Lambda=\mathcal J_\lambda,
    \qquad
    \Exp^{\mathrm{cw}}_\Lambda=\Exp^{\mathrm{cw}}_\lambda.
  \]
\end{thm}

\begin{proof}
  The successive Euler-power flows define
  \[
    F:\Spf R_S\llbracket s_1,\ldots,s_r\rrbracket
    \longrightarrow\mathcal B_X,
    \qquad
    F(u,s)=\varphi_{k_r}^{s_r}\circ\cdots\circ\varphi_{k_1}^{s_1}(u),
  \]
  with differential along $s=0$ given by
  \[
    dF_{(u,0)}\left(\xi+\sum_j a_j\partial_{s_j}\right)
    =
    \xi+\sum_j a_jE^{k_j}(u).
  \]
  After localization along $S$, condition
  \eqref{eq:formal-subscheme-spanning} and the formal implicit function
  theorem identify this chart with the normal formal neighborhood
  along $S$ over $K_S$. 
  By Lemma~\ref{lem:jordan-flow}, the eigenvalue multiplicities and Jordan partitions are preserved along the successive flows.

  Extend the coefficient field to $k_S$ and put
  $A=k_S\llbracket s_1,\ldots,s_r\rrbracket$.
  The factor of the characteristic polynomial belonging to $\lambda$ is
  coprime to its complementary factor over $A$. The proof of
  Proposition~\ref{prop:euler-spectral-block} therefore applies over
  $A$: the Sylvester operators are invertible over this ring, and the
  resulting gauge $G(s,z)\in I+z\End_A(H_X\otimes_\CC A)[[z]]$ makes $\mathcal H_\lambda[[z]][z^{-1}]$, where $\mathcal H_\lambda$ is the corresponding spectral summand of $H_X\otimes_\mathbb{C}A$, a differential summand in
  the transformed frame, preserved by the connection. After extension of coefficients, it is isomorphic to
  $\mathcal M_\Lambda$.

  On this summand write the transverse connections as
  $\partial_{s_a}+D_a(s,z)$ and solve
  \[
    \partial_{s_a}M_\lambda=-D_aM_\lambda,
    \qquad
    M_\lambda|_{s=0}=I.
  \]
  The coefficient recursion from the proof of
  Lemma~\ref{prop:formal-slice-solution-extension} applies also
  to these Laurent-series matrices: each fixed $s$-coefficient
  involves only finitely many products of Laurent series. Thus
  \[
    M_\lambda,\ M_\lambda^{-1}
    \in\End_{k_S}(\mathcal H_\lambda|_{s=0})((z))
      \llbracket s_1,\ldots,s_r\rrbracket.
  \]
  Flatness again gives
  $M_\lambda^{-1}\nabla_zM_\lambda=\nabla_z|_{s=0}$.
  Multiplication by $M_\lambda$ therefore transports the
  solution basis without changing its exponent classes.
\end{proof}

\begin{rem}
  The transport matrix lies in $k_S((z))\llbracket s\rrbracket$,
  but need not lie in
  $k_S\llbracket s\rrbracket[[z]][z^{-1}]$: its lowest $z$-order may
  tend to $-\infty$ with the $s$-degree. Thus the coefficientwise
  comparison alone does not identify the HLT factors over $k_X((z))$.
  For example, let $p,\lambda,\mu\in\CC$ and consider the flat connection
  \[
    \nabla_s=\partial_s+\frac pz,
    \qquad
    \nabla_z=\partial_z-\frac{\lambda+ps}{z^2}+\frac\mu z.
  \]
  Parallel transport from $s=0$ gives
  \[
    e^{-(\lambda+ps)/z}z^{-\mu}
    =
    e^{-\lambda/z}z^{-\mu}
    \sum_{n\geq0}\frac{(-ps)^n}{n!z^n}.
  \]
  Every coefficient retains the class $[-\mu]$, but for $p\ne0$
  the transport matrix $e^{-ps/z}$ has no uniform lowest $z$-order.
\end{rem}

Let $K$ be an algebraically closed field of characteristic zero,
  $A=K\llbracket v_1,\ldots,v_c\rrbracket $, and $F=\Frac(A)$.
  Let $\mathcal M$ be a free finite-rank differential module over
  $A[[z]][z^{-1}]$, integrable with respect to $\partial_z$ and the
  $\partial_{v_a}$, whose $z$-connection is
  \[
    \nabla_z=\partial_z-\frac{U(v)}{z^2}+O(z^{-1}),\qquad
    U\in\Mat_m(A).
  \]
The following proposition allows us to compare regular factors, and in particular exponent classes, without the Euler-spanning hypothesis.
\begin{prop}
  \label{lem:non-split-unramified}
 Suppose $\det(T-U)=(T-\Lambda)^m$ with $\Lambda\in A$, and the
  $z$-module obtained by extending scalars to $\overline F((z))$ is
  of unramified exponential type.
  Then $\mathcal E^{\Lambda/z}\otimes\mathcal M$, with the twist
  taken in all directions, admits a normal form
  \begin{equation}
    \label{eq:non-split-logarithmic-form}
    \nabla_z=\partial_z+\frac Cz,
    \qquad \nabla_{v_a}=\partial_{v_a},
    \qquad C\in\Mat_m(K),
  \end{equation}
  over $A[[z]][z^{-1}]$.
 Consequently the regular factors over $\overline F((z))$ and at
  $v=0$ are isomorphic after extension of coefficients. In particular,
  \begin{equation}
    \label{eq:non-split-exponents}
    \Exp_{\Lambda}=\Exp_{v=0}.
  \end{equation}
\end{prop}

\begin{proof}
 By exponential type and $\det(T-U)=(T-\Lambda)^m$, the only
  exponential factor over $\overline F((z))$ is $\mathcal E^{-\Lambda/z}$.
  Thus the scalar twist is regular singular over this field.
  Its only polar divisor over $K\llbracket v_1,\ldots,v_c,z\rrbracket$
  is $z=0$. Consequently
  \cite[Proposition~2.2.10 and Theorem~4.1.4]{KedlayaGoodFormalI},
  specifically \textup{(c)}$\Rightarrow$\textup{(b)}, gives the constant normal form
  \eqref{eq:non-split-logarithmic-form}.
  The same matrix $C$ represents the regular factors over
  $\overline F((z))$ and at $v=0$, proving
  \eqref{eq:non-split-exponents}.
\end{proof}

\begin{cor}
  Suppose a relative summand of the quantum connection satisfies the
  hypotheses of Proposition~\ref{lem:non-split-unramified}, and let
  $\mathcal M_\Lambda$ be its extension to $\overline F((z))$.
  Define $\Exp^{\mathrm{cw}}_\Lambda$ by transporting an HLT solution
  basis in $z$ at $v=0$. Then
  \[
    \Exp^{\mathrm{cw}}_\Lambda
    =\Exp_{v=0}
    =\Exp_\Lambda
    \subset\CC/\mathbb Z.
  \]
\end{cor}

\begin{proof}
  This follows from Lemma~\ref{prop:formal-slice-solution-extension} and Proposition~\ref{lem:non-split-unramified}.
\end{proof}

\subsection{Blow-up decomposition}
\label{sec:blowup}

Let $\widetilde X$ be the symplectic blow-up of $X$ along a closed
symplectic submanifold $Z\subset X$ of complex codimension $r$, with
$r\geq2$.
We recall the blow-up theorem of Iritani and its symplectic extension. 
Let $\mathfrak q$ be the variable associated with a line in a fiber of the exceptional divisor. 
The superscript `$\mathrm{la}$' denotes certain extensions of the quantum
$D$-modules; see \cite{IritaniBlowups} for details.

\begin{thm}
  \label{thm:cai-blowup-decomposition}
  There is an invertible formal change of bulk variables
  \begin{equation}
    \widetilde\tau\longmapsto
    \bigl(\tau(\widetilde\tau),
    \varsigma_0(\widetilde\tau),\ldots,
    \varsigma_{r-2}(\widetilde\tau)\bigr)
  \end{equation}
  such that
  \begin{equation}
    (H_{\widetilde X},\star_{\widetilde\tau})
    \simeq
    (H_X,\star_{\tau(\widetilde\tau)})
    \oplus\bigoplus_{j=0}^{r-2}
    (H_Z,\star_{\varsigma_j(\widetilde\tau)}),
  \end{equation}
  and
  \begin{equation}
    C_{E_{\widetilde X}}
    \longleftrightarrow
    C_{E_X}\oplus\bigoplus_{j=0}^{r-2}C_{E_Z}.
    \label{eq:blowup-Euler-operator}
  \end{equation}
  Moreover, this decomposition lifts to an isomorphism
  \begin{equation}
    \Psi:\QDM(\widetilde X)^{\mathrm{la}}
    \xrightarrow{\sim}
    \tau^*\QDM(X)^{\mathrm{la}}
    \oplus\bigoplus_{j=0}^{r-2}
    \varsigma_j^*\QDM(Z)^{\mathrm{la}}
    \label{eq:blowup-qdm}
  \end{equation}
  intertwining the full quantum connections, including the
  $z$-direction.
\end{thm}

For smooth projective blow-ups, this follows from \cite[Theorem~5.18]{IritaniBlowups}. 
The symplectic statement is
\cite[Theorem~3.5]{CaiQuarticSymplecticIrrationality}.

\begin{prop}
  \label{prop:spectral-blowup}
  Pass to the algebraic closure of the fraction field of the common
  coefficient ring in Theorem~\ref{thm:cai-blowup-decomposition}.
  Suppose that \eqref{eq:blowup-Euler-operator} identifies the generalized $\widetilde{\Lambda}$-eigenspace of $C_{E_{\widetilde{X}}}$ with the generalized $\Lambda$-eigenspace of $C_{E_Y}$, where $Y=X$ or $Y=Z$.
  Then the formal atomic invariants of
  $\mathrm{At}(\widetilde X)_{\widetilde\Lambda}$ and
  $\mathrm{At}(Y)_\Lambda$ agree.
\end{prop}

\begin{proof}
  Equation~\eqref{eq:blowup-Euler-operator} identifies the Euler
  operators on the corresponding generalized eigenspaces, preserving
  $m_\Lambda$ and $\mathcal J_\Lambda$.

  For the connection data, the essential point is the Laurent bound
  in $z$. By \cite[Remark~1.5 and Theorem~5.18]{IritaniBlowups},
  $\Psi$ and $\Psi^{-1}$ lie in a graded ring of the form $B[[z]]$.
  They therefore give inverse meromorphic gauges over
  $\overline{\Frac(B)}((z))$. The symplectic comparison has the same
  property by
  \cite[Theorem~3.5]{CaiQuarticSymplecticIrrationality}.
  The formal decomposition of
  Proposition~\ref{prop:euler-spectral-block} then identifies the
  corresponding $z$-connection summands.
  After a common ramification, uniqueness in
  Theorem~\ref{thm:levelt-turrittin} identifies, up to permutation,
  \[
    \varphi_{\widetilde\Lambda,a}^{\widetilde X}
    =
    \varphi_{\Lambda,a}^{Y},
    \qquad
    R_{\widetilde\Lambda,a}^{\widetilde X}
    \simeq
    R_{\Lambda,a}^{Y}.
  \]
  Hence the regular factors and their exponent multisets in
  $\CC/\mathbb Z$ are preserved. 
\end{proof}

\section{Low-dimensional blow-up centers}
\label{sec:obstruction}

In this section we describe the atoms of symplectic manifolds in complex dimensions 1 and 2.

\subsection{Curve atoms}
\begin{prop}
  \label{prop:complete-curve-atom-list}
  The formal atomic invariants of curves are listed below:
  \[
    \begin{array}{c|c|c|c}
      g & m & \mathcal J & \Exp\\
      \hline
      0 & 1+1 & (1)\oplus(1) & \{0\},\{0\}\\[1mm]
      1 & 2 & (1,1) & \{\tfrac12,\tfrac12\}\\[1mm]
      \geq2 & 2 & (2) & \{\tfrac12,\tfrac12\}
    \end{array}
  \]
\end{prop}

\begin{proof}
Let $\Sigma_g$ be a smooth connected curve of genus $g$.
Let $1\in H^0(\Sigma_g)$ be the unit and $p=[\mathrm{pt}]\in H^2(\Sigma_g)$ the point class. 
Let $Q$ be the Novikov variable with the divisor variable absorbed. 
For $\PP^1$, one has $p\star p=Q$, so generic Euler
multiplication has the two distinct eigenvalues
$\tau^0\pm2\sqrt{Q}$.  
Each atom has the exponent class zero.
For $g\geq1$, $\pi_2(\Sigma_g)=0$, hence
$p\star p=0$ and
\begin{equation*}
  C_E=\tau^0I+(2-2g)N,
  \qquad
  N=\begin{pmatrix}0&0\\1&0\end{pmatrix}.
\end{equation*}
After removing $\mathcal E^{-\tau^0/z}$, the shearing
$\diag(1,z^{-1})$ gives
\begin{equation*}
  \partial_z+\frac1z\left(-\frac12I-(2-2g)N\right).
  \label{eq:curve-regular-connection}
\end{equation*}
This has residue $-\frac12I-(2-2g)N$, so the two formal solution exponent classes are $1/2$.
\end{proof}

\subsection{Surface atoms}
\label{subsec:surface-atoms}

By a symplectic surface, we mean a closed symplectic manifold of complex dimension 2.
\begin{thm}[\cite{LiSympSurf,LiBettiNum,ZhangConeThm}]
    The following holds for symplectic surfaces:
	\begin{enumerate}
		\item Every symplectic surface is symplectically birational to a minimal one (see \cite{LiSympSurf} for the definition of minimal).
		\item Let $(S,\omega)$ be a minimal symplectic surface. Write $K_\omega=-c_1(TS)$. Define its Kodaira dimension by \[\kappa(S):=\begin{cases}
			-\infty & \textrm{if } K_\omega\cdot[\omega]<0 \textrm{ or } K_\omega^2<0,\\
			0 & \textrm{if } K_\omega\cdot[\omega]=0 \textrm{ and } K_\omega^2=0,\\
			1 & \textrm{if } K_\omega\cdot[\omega]>0 \textrm{ and } K_\omega^2=0,\\
			2 & \textrm{if } K_\omega\cdot[\omega]>0 \textrm{ and } K_\omega^2>0.
		\end{cases}\] Define the Kodaira dimension of a symplectic surface to be the Kodaira dimension of any of its minimal models. Then the Kodaira dimension is a well-defined symplectic birational invariant.
		\item Any symplectic surface with $\kappa=-\infty$ is symplectically birational to $\PP^1\times C$ for some curve $C$, and is necessarily algebraic.
		\item A minimal symplectic surface $S$ with $\kappa(S)\ge0$ is $K_\omega$-nef, i.e., $K_\omega\cdot A\ge0$ for any class $A$ representing a $J$-holomorphic curve for generic compatible $J$.
		\item A minimal symplectic surface $S$ with $\kappa(S)=0$ has torsion $K_S$. It satisfies $b_2(S)\le22$, with equality realized by the K3 surface.
	\end{enumerate}
\end{thm}
\begin{proof}
	This follows from \cite[Section 2]{LiSympSurf}, \cite[Theorem 1.1]{LiBettiNum} and \cite[Lemma 2.1]{ZhangConeThm}; note that we only need the statement of \cite[Lemma 2.1]{ZhangConeThm} for generic compatible $J$.
\end{proof}

\begin{rem}
	Symplectic surfaces with $\kappa\ge0$ can be non-algebraic \cite{GompfSympSum}.
\end{rem}

\begin{prop}
  \label{prop:surface-atom-list}
  The atoms of minimal symplectic surfaces are given in
  Table~\ref{tab:projective-surface-invariants}. A point blow-up adds one
  atom $(1,(1),\{0\})$.
\end{prop}

\begin{table}[H]
  \centering
  \small
  \setlength{\tabcolsep}{3pt}
  \renewcommand{\arraystretch}{1.16}
  \caption{Atoms of minimal symplectic surfaces.
  The invariants listed refer to each atom.}
  \label{tab:projective-surface-invariants}
  \begin{tabularx}{\textwidth}{@{}c>{\raggedright\arraybackslash}Xcccc@{}}
    \toprule
    $\kappa(S)$ & $S$ & Atoms & $m_\Lambda$
    & $\mathcal J_\Lambda$ & $\Exp_\Lambda$\\
    \midrule
    $-\infty$ & $\PP^2$ & $3$ & $1$ & $(1)$ & $\{0\}$\\
    $-\infty$ & $\mathbb \PP^1\times\PP^1$
    & $4$ & $1$ & $(1)$ & $\{0\}$\\
    $-\infty$ & Ruled over a curve of genus $1$
    & $2$ & $2$ & $(1,1)$ & $\{(1/2)^2\}$\\
    $-\infty$ & Ruled over a curve of genus $g\geq2$
    & $2$ & $2$ & $(2)$ & $\{(1/2)^2\}$\\
    \midrule
    $0$ & Symplectic Calabi--Yau ($b_2\le22$) & $1$ & $b_2+2$ & $(2,1^{b_2})$ & $\{0^{b_2+2}\}$\\
    $1$ & Minimal symplectic surface with $\kappa=1$ & $1$ & $b_2+2$
    & $(2,2,1^{b_2-2})$ & $\{0^{b_2+2}\}$\\
    $2$ & Minimal symplectic surface with $\kappa=2$ & $1$ & $b_2+2$
    & $(3,1^{b_2-1})$ & $\{0^{b_2+2}\}$\\
    \bottomrule
  \end{tabularx}
\end{table}

For a symplectic surface $S$, we write $\tau^0$ for the variable dual to the unit, and $\tau^1$ for the variable dual to $[\mathrm{pt}]$. 
The Euler element is \[E_S=\tau^01+c_1(S)-\tau^1[\mathrm{pt}].\]

\begin{proof}
  Every symplectic surface is obtained from a minimal model by point
  blow-ups.  A point contributes the connection
  $\partial_z-\tau^0z^{-2}$, so Proposition~\ref{prop:spectral-blowup} gives the
  additional atom $(1,(1),\{0\})$. 
  
 The three rank-one atoms of $\PP^2$ have exponent zero.
Applying \cite[Theorem~1.7]{IritaniKotoProjectiveBundles} to
$\PP^1\times\PP^1=\PP_{\PP^1}(\OO_{\PP^1}^{\oplus2})$
gives two copies of the $\PP^1$ quantum connection.
Hence there are four rank-one atoms, each with exponent class zero.

  A ruled surface over a curve $C$ of genus $g \geq 1$ is birational to $C\times\PP^1$; both have $b_2=2$,
  so their common resolution adds the same number of point atoms.  On $C\times\PP^1$, let $\eta$ and $h$ be the pull-backs of the point classes of $C$ and
  $\PP^1$, respectively, and write $\alpha=\eta\cup h$. Let $q$ be the Novikov variable with the divisor variables absorbed. Introduce $\widetilde h=h-\frac{\tau^1}{2}\alpha$. The big quantum product satisfies
  \[
    \eta\star\eta=0,\qquad
    \widetilde h\star\widetilde h=q.
  \] 
  The Euler vector field is \[E=\tau^01+(2-2g)\eta+2\widetilde h.\]
  The two Euler blocks are therefore
  $(\tau^0\pm2\sqrt q)I+(2-2g)J_2(0)$.
  For the quantum connection, we apply the projective-bundle decomposition of \cite{IritaniKotoProjectiveBundles} to
  $C\times\PP^1=\PP_C(\OO_C^{\oplus2})$; its dual bundle is globally
  generated, as required in \cite[Theorems~1.7 and~5.1]{IritaniKotoProjectiveBundles}.
 Each connection summand is identified with the curve connection after the
  invertible change of bulk variables and a meromorphic change of frame.
  Thus Proposition~\ref{prop:complete-curve-atom-list} gives exponent class
  $1/2$ on both branches. This gives the two ruled rows.

  For $\kappa(S)\geq0$, since $K_\omega$ is nef with respect to generic compatible almost complex structures, $S$ has no nontrivial genus 0 Gromov--Witten invariants. Write $A=c_1(S)\cup-$, $B=[\mathrm{pt}]\cup-$, and $N=A-\tau^1B$.  Then
  \[
    C_E=\tau^0I+N,\qquad N^3=0,\qquad
    (\rk N,\rk N^2)=
    \begin{cases}
      (1,0),&\kappa(S)=0,\\
      (2,0),&\kappa(S)=1,\\
      (2,1),&\kappa(S)=2.
    \end{cases}
  \]
  These ranks give the remaining Jordan partitions. Since $[\mu_S,A]=A$,
  $[\mu_S,B]=2B$, and $AB=BA=B^2=0$, the substitution $Y=e^{-\tau^0/z}z^{-\mu_S}e^{-\tau^1zB}Z$ gives
  \begin{equation}
    \label{eq:surface-regular-gauge}
    \partial_zZ-\frac AzZ=0
  \end{equation}
  with the fundamental solution $Z=z^A$.
  After removing the scalar exponential, the gauge and its inverse are
  meromorphic, since $\mu_S$ has integral eigenvalues. Thus the Euler
  partition is determined by $A-\tau^1B$, whereas the regular module is determined
  by $A$. In particular, the latter is trivial for $\kappa(S)=0$, when
  $A=0$. Since $A$ is nilpotent, every exponent class is zero.
\end{proof}

\begin{rem}
    Similar computations in the algebraic case appear in \cite{GyengeSzaboBlowups}.
\end{rem}

After extending the coefficients to an algebraically closed field $k$, write 
$\mathcal U_j=\left(k\llaur z\rlaur ^j,\partial_z-J_j(0)/z\right)$,
$\mathcal U_1=\mathbf1=(k\llaur z\rlaur ,\partial_z)$, and
$\mathcal E^{\tau^0/z}=(k\llaur z\rlaur ,\partial_z+\tau^0/z^2)$.
For $\kappa(S) \geq 0$, let $\mathcal{M}_S$ denote the unique atom of $S$ and set $m_S=b_2(S)+2$. By \eqref{eq:surface-regular-gauge}, we have the following invariants 
\begin{equation}
  \label{eq:surface-connection-types}
  \begin{array}{c|c|c}
    \kappa(S)&\mathcal J_S&\mathcal E^{\tau^0/z}\otimes\mathcal M_S\\
    \hline
    0&(2,1^{m_S-2})&\mathbf1^{\oplus m_S}\\
    1&(2,2,1^{m_S-4})&\mathcal U_2^{\oplus2}\oplus
    \mathbf1^{\oplus(m_S-4)}\\
    2&(3,1^{m_S-3})&\mathcal U_3\oplus\mathbf1^{\oplus(m_S-3)}
  \end{array}.
\end{equation}
Since $\mathcal U_j$ is indecomposable for $j>1$, the maximal
rank of a trivial differential direct summand is $m_S-4$ for
$\kappa(S)=1$ and $m_S-3$ for $\kappa(S)=2$.

\section{Fano weighted complete intersection fourfolds}
\label{sec:ci-fourfolds}

Let $X=X_{d_1,\ldots,d_c}\subset\PP(a_0,\ldots,a_{4+c})$ be a
smooth well-formed Fano complete intersection, which is not an intersection with
a linear cone. In other words, $\gcd(a_0, \dots, \hat{a_i}, \dots, a_{4+c})=1$ for every $i$, and $d_j \notin \{ a_0, \cdots, a_{4+c}\}$ for every $j$. Put
\[
  H=c_1(\OO_X(1)),\quad
  \iota=\sum_i a_i-\sum_jd_j,\quad
  d_X=\frac{\prod_jd_j}{\prod_i a_i}=\int_XH^4,\quad
  N=\dim H^4_{\pr}(X).
\]
Such an $X$ is disjoint from the singular locus of the ambient space
and is quasismooth by \cite[Proposition~2.11 and Corollary~2.14]{PrzyjalkowskiShramovBounds}.
Thus $H$ is a Cartier class. The numbers $\iota$ and $d_X$ are its
Fano index and degree, and $-K_X=\iota H$. There are 22 families excluding $\PP^4$, which we list in Table \ref{ci:tab:irrationality}; we use the numbering of \cite[Table~1]{PrzyjalkowskiShramovBounds}. 

\begin{rem}
    The exclusion of linear cones causes no loss at the level of deformation families.
Indeed, special linear cone members deform to members for which the corresponding
coordinate can be eliminated; for example, $X_{2,4}\subset\PP(1^6,2)$ deforms to
a quartic $X_4\subset\PP^5$; see \cite[Remark~5.2]{PrzyjalkowskiShramovBounds}.
\end{rem}

\begingroup
\setlength{\tabcolsep}{4pt}
\renewcommand{\arraystretch}{1.1}
\begin{longtable}{@{}rlrr>{\raggedright\arraybackslash}p{.28\textwidth}c@{}}
  \caption{Smooth Fano weighted complete intersection fourfolds excluding
  $\PP^4$. Irrationality means every smooth member unless otherwise stated.
  The last column records $\sigma(X)=4-2\iota/\max_j d_j$.}
  \label{ci:tab:irrationality}\\
  \toprule
  No.&$X$&$\iota$&$N$&Conclusion&$\sigma(X)$\\
  \midrule\endfirsthead
  \toprule
  No.&$X$&$\iota$&$N$&Conclusion&$\sigma(X)$\\
  \midrule\endhead
  \bottomrule\endfoot
  1&$X_{6,6}\subset\PP(1^3,2^2,3^2)$&1&716&
  irrational, Theorem~\ref{ci:thm:index-one}&$11/3$\\
  2&$X_{10}\subset\PP(1^4,2,5)$&1&2624&
  irrational, Theorem~\ref{ci:thm:index-one}&$19/5$\\
  3&$X_{4,6}\subset\PP(1^4,2^2,3)$&1&813&
  irrational, Theorem~\ref{ci:thm:index-one}&$11/3$\\
  4&$X_8\subset\PP(1^5,4)$&1&2101&
  irrational, Theorem~\ref{ci:thm:index-one}&$15/4$\\
  5&$X_6\subset\PP(1^5,2)$&1&1042&
  irrational, Theorem~\ref{ci:thm:index-one}&$11/3$\\
  6&$X_{4,4}\subset\PP(1^5,2^2)$&1&527&
  irrational, Theorem~\ref{ci:thm:index-one}&$7/2$\\
  7&$X_{2,6}\subset\PP(1^6,3)$&1&1303&
  irrational \cite[Section~0]{PukhlikovDoubleQuadric};
  see Remark~\ref{ci:rem:double-quadric}&$11/3$\\
  8&$X_5\subset\PP^5$&1&820&
  irrational, Theorem~\ref{ci:thm:index-one}&$18/5$\\
  9&$X_{3,4}\subset\PP(1^6,2)$&1&505&
  irrational, Theorem~\ref{ci:thm:index-one}&$7/2$\\
  10&$X_{2,4}\subset\PP^6$&1&547&
  irrational, Theorem~\ref{ci:thm:index-one}&$7/2$\\
  11&$X_{3,3}\subset\PP^6$&1&364&
  irrational, Theorem~\ref{ci:thm:index-one}&$10/3$\\
  12&$X_{2,2,3}\subset\PP^7$&1&319&
  irrational, Theorem~\ref{ci:thm:index-one}&$10/3$\\
  13&$X_{2,2,2,2}\subset\PP^8$&1&219&
  irrational, Theorem~\ref{ci:thm:index-one}&$3$\\
  14&$X_6\subset\PP(1^5,3)$&2&521&
  irrational, Theorem~\ref{ci:thm:index-two}&$10/3$\\
  15&$X_4\subset\PP^5$&2&183&
  irrational, Theorem~\ref{ci:thm:index-two}&$3$\\
  16&$X_{2,3}\subset\PP^6$&2&85&
  irrational, Theorem~\ref{ci:thm:index-two}&$8/3$\\
  17&$X_{2,2,2}\subset\PP^7$&2&43&
  rational members exist
  \cite[Theorem~1]{HassettPirutkaTschinkelThreeQuadrics};
  see Remark~\ref{ci:rem:three-quadrics}&$2$\\
  18&$X_6\subset\PP(1^4,2,3)$&3&208&
  irrational, Theorem~\ref{ci:thm:index-three}&$3$\\
  19&$X_4\subset\PP(1^5,2)$&3&61&
  irrational, Theorem~\ref{ci:thm:index-three}&$5/2$\\
  20&$X_3\subset\PP^5$&3&22&
  rational members exist
  \cite{BolognesiRussoStaglianoRationalCubics};
  see Remark~\ref{ci:rem:cubic-fourfold}&$2$\\
  21&$X_{2,2}\subset\PP^6$&3&7&
  rational \cite[Table~2]{PrzyjalkowskiShramovBounds};
  see Remark~\ref{ci:rem:small-monodromy}&$1$\\
  22&$X_2\subset\PP^5$&4&1&
  rational \cite[Table~2]{PrzyjalkowskiShramovBounds}&$0$\\
\end{longtable}
\endgroup

Throughout this section, $X$ will always be a smooth member of Nos.~1--20 excluding No.~7.

\subsection{Monodromy}
\label{ci:subsec:weighted-monodromy}

Let $S_{\mathbf d,\mathbf a}$ be the open parameter space of smooth
complete intersections of multidegree $\mathbf d$ in
$\PP(\mathbf a)$, and let $\rho_{\pr}$ denote the representation of $\pi_1(S_{\mathbf d,\mathbf a})$ on
$V=H^4_{\pr}(X,\CC)$. By Zariski closure, we mean Zariski closure of
the image of this representation.

\begin{lem}
  \label{ci:lem:reflection-density}
  Let $V_{\mathbb R}$ carry a nondegenerate indefinite symmetric form,
  with $\dim V_{\mathbb R}\ge3$. Suppose that
  $\Gamma\subset O(V_{\mathbb R})$ acts irreducibly on
  $V=V_{\mathbb R}\otimes\CC$ and is generated by reflections and
  unipotent transformations, with at least one reflection. Then
  $\overline\Gamma^{\,\mathrm{Zar}}=O(V)$.
\end{lem}

\begin{proof}
  Put $G=\overline{\Gamma}^{\,\mathrm{Zar}}$, and let $G^\circ$ denote
  the identity component of $G$ in the Zariski topology. Since $G$ acts
  irreducibly on $V$, its unipotent radical is trivial, so $G$ is
  reductive. Moreover, $G$ cannot be finite: averaging a positive
  definite real form over $G$ and applying Schur's lemma would force the
  given $G$-invariant symmetric form to be definite. Hence $G^\circ$ is
  nontrivial.

  We first show that $G^\circ$ acts irreducibly. Since
  $G^\circ\triangleleft G$, the $G^\circ$-isotypic components of $V$ are
  permuted transitively by $G$. Every unipotent element of $G$ lies in
  $G^\circ$ and hence preserves each component. Since
  $\rk(r-I)=1$ for a reflection $r$, a reflection can permute two
  distinct components only if they are one-dimensional.

  Suppose first that the isotypic components have dimension greater than
  one. Since $\Gamma$ is generated by reflections and unipotents,
  irreducibility forces a single isotypic component, say
  $V\simeq U\otimes M$. A reflection normalizing the $G^\circ$-action
  has the form $A\otimes B$: choose an intertwiner on $U$ for its
  conjugation action, then apply Schur's lemma to the multiplicity space. Since $(A\otimes B)^2=I$, after reciprocal
  rescaling we may assume $A^2=B^2=I$. If
  $\dim U,\dim M\geq2$, the eigenvalues $a_i b_j$ of $A\otimes B$
  cannot consist of exactly one $-1$ and all remaining $1$, since
  $(a_i b_j)(a_k b_l)=(a_i b_l)(a_k b_j)$. Hence $\dim M=1$.
  The alternative $\dim U=1$ would make $G^\circ$ a connected scalar
  subgroup of $O(V)$, hence trivial.

  Suppose instead that all isotypic components are lines. Then $G^\circ$
  is a torus. The trivial weight cannot occur, since its isotypic
  component would be $G$-invariant. Preservation of the symmetric form
  pairs the weight spaces as $V_\chi\oplus V_{\chi^{-1}}$. A reflection cannot
  permute two distinct such planes, since this would give
  $\rk(r-I)\geq2$, while $G^\circ$ contains no nontrivial unipotents.
  Thus irreducibility would force $\dim V=2$, contrary to assumption.
  Therefore $G^\circ$ acts irreducibly on $V$. By Schur's lemma its
  connected center acts by scalars; since $G^\circ\subset O(V)$, this
  center is trivial. Hence $G^\circ$ is semisimple.

  Choose a reflection $r_\delta\in G$. For $g\in G^\circ$, we have
  $r_{g\delta}=g r_\delta g^{-1}$, so
  $r_\delta r_{g\delta}\in G^\circ$, and this product fixes
  $\delta^\perp\cap(g\delta)^\perp$. Since $G^\circ$ acts irreducibly,
  the orbit $G^\circ\delta$ spans $V$. If
  $g\mapsto(\delta,g\delta)$ were constant, the covector
  $(\delta,-)$ would be $G^\circ$-invariant, a contradiction. We may therefore
  choose $g$ such that $r_\delta r_{g\delta}$ is conjugate to
  $\diag(\lambda,\lambda^{-1},1,\ldots,1)$ with $\lambda$ not a root of
  unity.

  By \cite[Theorem~4.1]{AudibertDoubaLeeMarquis}, a connected semisimple
  group acting faithfully and irreducibly and containing such an element
  is, in its defining representation, one of $SL(V)$, $SO(V)$, or
  $Sp(V)$. Since $G^\circ\subset O(V)$, only the orthogonal case is
  possible, and hence $G^\circ=SO(V)$. Finally, $G$ contains a reflection
  of determinant $-1$, so $G=O(V)$.
\end{proof}

\begin{prop}
  \label{ci:prop:weighted-monodromy}
  For families Nos.~1--20 in Table~\ref{ci:tab:irrationality}, except No.~7,
  \begin{equation*}
    \label{ci:eq:weighted-monodromy}
    \overline{\rho_{\pr}(\pi_1(S_{\mathbf d,\mathbf a}))}
    ^{\,\mathrm{Zar}}=O(V).
  \end{equation*}
\end{prop}

\begin{proof}
  For the ordinary projective complete intersections in the table, this follows from the work of Deligne \cite{SGA7II,DeligneWeilII}. 
  We treat the genuinely weighted families
  other than No.~7.

  Fix general lower-degree equations and vary the last equation on the
  fivefold $Y$ in Table~\ref{ci:tab:monodromy-fivefolds}.

  \begin{table}[htbp]
    \centering\small
    \renewcommand{\arraystretch}{1.15}
    \caption{Auxiliary fivefolds for weighted monodromy. The degree specifies
    the very ample system; quotient types are listed without multiplicities.}
    \label{ci:tab:monodromy-fivefolds}
    \begin{tabular}{@{}rlrl@{}}
      \toprule
      No.&$Y$&Degree&Quotient types\\
      \midrule
      1&$X_6\subset\PP(1^3,2^2,3^2)$&6&$\frac12(1^5),\ \frac13(1^3,2^2)$\\
      2&$\PP(1^4,2,5)$&10&$\frac12(1^5),\ \frac15(1^4,2)$\\
      3&$X_4\subset\PP(1^4,2^2,3)$&6&$\frac12(1^5),\ \frac13(1^3,2^2)$\\
      4&$\PP(1^5,4)$&8&$\frac14(1^5)$\\
      5&$\PP(1^5,2)$&6&$\frac12(1^5)$\\
      6&$X_4\subset\PP(1^5,2^2)$&4&$\frac12(1^5)$\\
      9&$X_3\subset\PP(1^6,2)$&4&$\frac12(1^5)$\\
      14&$\PP(1^5,3)$&6&$\frac13(1^5)$\\
      18&$\PP(1^4,2,3)$&6&$\frac12(1^5),\ \frac13(1^4,2)$\\
      19&$\PP(1^5,2)$&4&$\frac12(1^5)$\\
      \bottomrule
    \end{tabular}
  \end{table}

  Each $Y$ is quasismooth with only the isolated cyclic quotient
  singularities displayed. The indicated degree monomials give a
  closed embedding: on each weighted affine chart, invariant monomial
  generators are ratios of monomials of the indicated degree.
  Restriction of weighted polynomials gives the complete hyperplane system.
  Weighted Lefschetz \cite[Proposition~6]{DimcaWeightedMonodromy} gives
  \[
    H^4(Y,\CC)=\CC H^2.
  \]
  These quotient points are rational homology manifold points, so
  $H^4(Y)\to H^4(Y_{\reg})$ is an isomorphism. Consequently the image
  of $H^4(\widetilde Y)$ in $H^4(X)$, for a resolution
  $\widetilde Y\to Y$, equals the image of $H^4(Y)$: exceptional
  classes restrict to zero because $X$ avoids the singular points.
  The invariant-cycle theorem therefore identifies the invariant part with
  $\CC H^2$, and the vanishing cohomology is
  \[
    H^4(X)_{van}=H^4_{\pr}(X)=V.
  \]
  By \cite[Theorem~3.1 and Corollary~3.7]{DiGennaroFranco}, the induced
  monodromy on $V$ is irreducible, also over $\CC$, and the meridians
  about the ordinary nodal locus act by Picard--Lefschetz reflections.
  Indeed, the standard vanishing cycles form one orbit and span $V$;
  the usual reflection proof of irreducibility works over $\CC$ as well.

  It remains to consider a discriminant component consisting of sections
  through a quotient point
  \[
    0\in \CC^5/\mu_r,
    \qquad
    \zeta\cdot(x_1,\ldots,x_5)
    =
    (\zeta^{b_1}x_1,\ldots,\zeta^{b_5}x_5).
  \]
  For a general section through such a point, the local equation has the form
  \[
    f=f_r+\text{(terms of weighted degree $>r$)},
  \]
  where $f_r$ has an isolated critical point. The required principal
  parts occur in the global systems above, as shown by the following
  choices; having an isolated critical point is an open condition:
  \[
    \begin{array}{c|l}
      \text{type}&f_r\\ \hline
      \frac12(1^5)&\sum_{j=1}^5x_j^2\\
      \frac13(1^5)&\sum_{j=1}^5x_j^3\\
      \frac13(1^4,2)&yx_1+x_2^3+x_3^3+x_4^3\\
      \frac13(1^3,2^2)&y_1x_1+y_2x_2+x_3^3\\
      \frac14(1^5)&\sum_{j=1}^5x_j^4\\
      \frac15(1^4,2)&x_1y^2+x_1^5+x_2^5+x_3^5+x_4^5.
    \end{array}
  \]
  Consider
  \[
    f_\epsilon(x)
    =
    \epsilon^{-r}
    f(\epsilon^{b_1}x_1,\ldots,\epsilon^{b_5}x_5),
    \qquad
    f_0=f_r.
  \]
  The isolated weighted principal part makes this a deformation of
  constant Milnor number; see \cite[I, Corollary~2.18]{GreuelLossenShustin}.
  {After shrinking the parameter space if necessary, choose a common
$\mu_r$-invariant Milnor ball $B_\delta$ and a disk
$D_\eta\subset\CC$ such that
\[
  f_\epsilon:
  B_\delta\cap f_\epsilon^{-1}(D_\eta^*)
  \longrightarrow D_\eta^*
\]
is a Milnor fibration for every $\epsilon$. Parallel transport of the
middle cohomology of the Milnor fibers is $\mu_r$-equivariant.}
  
For $f_r$,
  the geometric monodromy is
  \[
    (x_1,\ldots,x_5)
    \longmapsto
    (\zeta^{b_1}x_1,\ldots,\zeta^{b_5}x_5),
    \qquad
    \zeta=e^{2\pi i/r},
  \]
  which is exactly the $\mu_r$-action defining the quotient; see
  \cite[Section~9]{MilnorSingularPoints}. The group acts freely on
  the Milnor fiber $M_r$, so transfer identifies its quotient's
  middle homology with $H_4(M_r,\CC)^{\mu_r}$. The local monodromy on
  this invariant part is the identity. Let $V_p$ be its image in
  $H_4(X,\CC)\simeq H^4(X,\CC)$. The variation map gives
  \[
    \im(T-I)\subset V_p\subset\Ker(T-I),
  \]
  and hence $(T-I)^2=0$.

  The fundamental group of the complement of the hyperplane discriminant is generated by
  meridians and their conjugates. Its codimension-one components are
  the nodal locus and the hyperplanes through the quotient points.
  Thus its monodromy is generated by reflections and the unipotents
  just obtained. There is a reflection because the nodal vanishing
  cycles span the nonzero space $V$.

  Finally, the primitive real form is indefinite, since
  \[
    h^{3,1}(X)>0,
    \qquad
    h^{2,2}_{\pr}(X)>0
  \]
  for all these families; see \cite[Table~1]{PrzyjalkowskiShramovBounds}.
  Lemma~\ref{ci:lem:reflection-density} therefore gives
  \[
    \overline{\rho_{\pr}(\pi_1)}^{\,\mathrm{Zar}}=O(V)
  \]
  for this hyperplane-section subfamily, and hence also for the full
  family. The smooth parameter space is connected, so parallel
  transport gives the same conclusion at every smooth member.
\end{proof}

\subsection{Monodromy reduction}

Cohomology has complex coefficients unless indicated otherwise. We write
\[
  H^{\ev}(X)=A\oplus V,\qquad
  A=\CC\langle 1,H,H_2,H_3,H_4\rangle,\qquad
  V=\Ker\bigl(H\cup-:H^4(X)\to H^6(X)\bigr).
\]
Here $H_k$ is the $k$-th power of $H$ under the classical cup product, with $H_0=1$ and $H_1=H$. The Poincar\'e form is
bilinear, $A,V$ are perpendicular, and
\[
  g_{ab}=(H_a,H_b)_X=d_X\delta_{a+b,4},\qquad
  (g^{ab})=(g_{ab})^{-1}.
\]
Choose an orthonormal basis $\phi_1,\ldots,\phi_N$ of $V$, and fix
$\phi\in V$ with $(\phi,\phi)_X=1$. Put $W_\phi=\phi^\perp\subset V$.

After extending the Novikov coefficients, we work over
\[
  K=\overline{\Frac(\CC\llbracket Q\rrbracket )},\qquad
  R=K\llbracket t^0,\ldots,t^4,u^1,\ldots,u^N\rrbracket ,\qquad
  \tau=t+u=\sum_at^aH_a+\sum_iu^i\phi_i.
\]
Consider the ideals $(t)=(t^0,\dots,t^4)$, $(u)=(u^1,\dots,u^N)$, and $I_\phi=(( u,w)_X:w\in W_\phi)$. Introduce the following rings:
\begin{equation*}
  \label{ci:eq:coefficient-rings}
  \begin{array}{c|c|c}
    \text{ring}&\text{presentation}&\text{bulk locus}\\
    \hline
    R_{\sm}&K&t=u=0\\
    R_{\mathrm{amb}}&R/(u)=K\llbracket t^0,\ldots,t^4\rrbracket &u=0\\
    R_{\mathrm{rad}}&R/I_\phi=K\llbracket t^0,\ldots,t^4,r\rrbracket &u=r\phi\\
    R_{\pr}&R/(t)=K\llbracket u^1,\ldots,u^N\rrbracket &t=0\\
    R_r&R/(t,I_\phi)=K\llbracket r\rrbracket &t=0,\ u=r\phi.
  \end{array}
\end{equation*}
Write $S_\bullet=\Spf R_\bullet$, $K_\bullet=\Frac(R_\bullet)$,
$k_\bullet=\overline{K_\bullet}$, and $k=\overline{\Frac(R)}$.

Set $\mathcal{H}=H^{\ev}(X)\otimes_\mathbb{C}R$. For a restriction ring $R_\bullet$ in the table , write $A_\bullet=A\otimes_\CC R_\bullet$,
$V_\bullet=V\otimes_\CC R_\bullet$, and $\mathcal H_\bullet=H^{\ev}(X)\otimes_\CC R_\bullet$.
For a summand $\mathcal H_i$, put
$\mathcal H_{i,\bullet}=\mathcal H_i\otimes_RR_\bullet$.
Subscripts on functions, vectors, products, and matrices denote the
same restriction. 
Finally, we write $\mathcal A_z=R\llbracket z\rrbracket [z^{-1}]$.

For $x \in H^{\ev}(X)$, let $\partial_x$ denote differentiation in the direction of the corresponding constant vector field. For $f\in R$, define its gradient and differential
\begin{equation*}
  \label{ci:eq:full-gradient}
  \grad f
  =\sum_{a,b=0}^4g^{ab}f_{t^b}H_a+\sum_{i=1}^Nf_{u^i}\phi_i,
  \qquad
  (\grad f,x)_X=df(x).
\end{equation*}
For an $R$-valued vector $w$, write $\partial_w f=df(w)$.
After restriction,
\begin{equation*}
  \label{ci:eq:restricted-differential}
  (( \grad f)_\bullet,w)_X=(df)_\bullet(w),\qquad
  (df)_\bullet(w)=d(f_\bullet)(w)\quad \text{ for }w\in TS_\bullet.
\end{equation*}
We identify $TS_{\mathrm{amb}}=A_{{\mathrm{amb}}},
TS_{\mathrm{rad}}=A_{\mathrm{rad}}\oplus R_{\mathrm{rad}}\phi,
TS_{\pr}=V_{{\pr}}, TS_r=R_r\phi$.

Introduce the variable $s=\frac{1}{2}(u,u)_X.$ Note that $s$ is quadratic in the flat primitive coordinates.
\begin{prop}
    \label{ci:prop:monodromy-invt}
    The genus-zero Gromov--Witten potential $\Phi$ of $X$ is $O(V)$-invariant, and has the form
    \begin{equation}
    \label{ci:eq:monodromy-inv}
        \Phi=\Phi(t,s).
    \end{equation}
    In particular,
    \begin{equation}
\label{ci:eq:prod-equiv}
    g(a\star_{t,u}b)=(ga)\star_{t,gu}(gb),\ \text{ for } g\in O(V). 
\end{equation}
\end{prop}
\begin{proof}
    Since Gromov--Witten invariants are deformation invariant, they are invariant under the monodromy action. Since they are multilinear in the insertions, they are moreover invariant under the Zariski closure of the monodromy action, which is $O(V)$ by Proposition~\ref{ci:prop:weighted-monodromy}. Then \eqref{ci:eq:monodromy-inv} follows from the first fundamental theorem of invariant theory for the orthogonal group.
\end{proof}

\begin{lem}
  \label{ci:lem:radial-testing}
  An $O(V)$-invariant formal tensor over $R$ is zero if and only if its
  restriction to $R_{\mathrm{rad}}$ is zero. The same holds for the restriction
  $R_{\pr}\to R_r$.
\end{lem}

\begin{proof}
  Expand in $t$ and in homogeneous polynomials in $u$.
  A homogeneous coefficient vanishing at $u=r\phi$ vanishes on all scalar multiples of the $O(V)$-orbit of $\phi$. 
  These include all $v\in V$ with $(v,v)_X\ne0$. 
  This locus is dense, so the coefficient is identically zero.
\end{proof}

Write
\begin{equation*}
  \label{ci:eq:invariant-ring}
  R_{\mathrm{inv}}=R^{O(V)}=R_{\mathrm{amb}}\llbracket s\rrbracket \subset R.
\end{equation*}
For a function $f=f(t,s)\in R_\mathrm{inv}$, write $f_a=\partial_{t^a}f$ and $f_s=\partial_sf$, and use analogous subscripts for higher derivatives. For a constant $v\in V$, the chain rule gives
\begin{equation}
  \label{ci:eq:primitive-chain-rule}
  \partial_v s=(u,v)_X,\qquad
  \partial_v\Phi=(u,v)_X\Phi_s,
\end{equation}
\begin{equation}
  \label{ci:eq:invariant-gradient}
  \begin{gathered}
    \grad f=\sum_{a,b=0}^4g^{ab}f_bH_a+f_su,\\
    (\grad f,\grad f)_X
    =\sum_{a,b=0}^4g^{ab}f_af_b+2sf_s^2.
  \end{gathered}
\end{equation}
Put $\Lambda=\Phi_s\in R_\inv$, $\Theta=\grad\Lambda$.
By \eqref{ci:eq:invariant-gradient},
\begin{equation}
  \label{ci:eq:Theta-gradient}
  \Theta=\grad\Lambda
  =\sum_{a,b=0}^4g^{ab}\Phi_{sb}H_a+\Phi_{ss}u\in\mathcal H.
\end{equation}
Recall that we fix $\phi\in V$ with $(\phi,\phi)_X=1$, and put $B=A\oplus\CC\phi,\ W=\phi^\perp\subset V.$

\begin{lem}
    There is an algebra homomorphism $\chi_\rad:B_\rad\to R_\rad$ such that 
    \begin{equation}
    \label{ci:eq:radial-character}
        b\star_\rad w=\chi_\rad(b)w,\quad\textrm{ for } b\in B_\rad, w\in W_\rad.
    \end{equation}
    Also, \begin{equation}
    \label{ci:eq:Theta-products}
        w_1\star_\rad w_2=(w_1,w_2)_X\Theta_\rad,\quad\textrm{ for } w_1,w_2\in W_\rad.
    \end{equation}
\end{lem}
\begin{proof}
  The orthogonal group $O(W)$ acts on $\mathcal{H}_\rad$ and fixes $B_\rad$ pointwise.
  Since $\star_\rad$ is $O(W)$-equivariant,
    \begin{gather*}
        B_\rad\star_\rad B_\rad\subset B_\rad,\\
        B_\rad\star_\rad W_\rad\subset W_\rad,\\
        W_\rad\star_\rad W_\rad\subset B_\rad.
    \end{gather*}
  Here the third inclusion follows from the fact that 
  \[\Hom_{O(W)}(\Sym^2W,W)=0.\]
  Now, \eqref{ci:eq:radial-character} follows from the fact that
  \[\End_{O(W)}(W)=\CC I.\]
  Similarly, since \[(\Sym^2W^\vee)^{O(W)}=\CC(-,-)_X,\]
  it follows that
  \[w_1\star_\rad w_2=(w_1,w_2)_X(\eta\star_\rad\eta)\]
  for any $\eta\in W_\rad$ with $(\eta,\eta)_X=1$.
  It remains to identify $\eta\star_\rad\eta$ with $\Theta_{\mathrm{rad}}$.
  Take $\eta$ to be constant and put $a_\eta=(u,\eta)_X$.
  By \eqref{ci:eq:primitive-chain-rule},
  \[
    \begin{aligned}
      \partial_\eta^2\Phi&=\Phi_s+a_\eta^2\Phi_{ss},\\
      \partial_\xi\partial_\eta^2\Phi
      &=\partial_\xi \Phi_s+2a_\eta(\partial_\xi a_\eta)\Phi_{ss}
      +a_\eta^2\partial_\xi \Phi_{ss}.
    \end{aligned}
  \]
  Now restrict to $u=r\phi$, where $a_\eta=0$. 
  For every constant vector $\xi$,
  \[
    (\eta\star_\rad\eta,\xi)_X
    =(\partial_\xi \Phi_s)_{\mathrm{rad}}
    =(\Theta_{\mathrm{rad}},\xi)_X.
  \]
  Since the pairing is nondegenerate, this proves \eqref{ci:eq:Theta-products}.
\end{proof}

Using the identification of vector fields with cohomology, write the Euler vector field as 
\[E=\iota H+\sum_{a=0}^4(1-a)t^aH_a-u,\]
and 
\[U=C_E=E\star(-):\mathcal{H}\to\mathcal{H}.\]
\begin{lem}
\label{ci:lem:Theta-identities}
    The vector field $\Theta$ has the following properties:
    \begin{gather}
        \Theta\star\Theta=0,\qquad (\Theta,1)_X=1,\label{ci:eq:Theta-identities}\\
        x\star\Theta=(\partial_x\Lambda)\Theta,\quad\textrm{ for } x\in\mathcal{H},\label{ci:eq:Theta-x}\\
        \chi_\rad(E)=\Lambda_\rad,\qquad
        U\Theta=\Lambda\Theta,\qquad U_{\rad}|_{W_{\rad}}=\Lambda_{\rad} I.\label{ci:eq:Theta-eigen}
    \end{gather}
\end{lem}

\begin{proof}
  First work over $R_{\mathrm{rad}}$. By \eqref{ci:eq:Theta-products}, for $\eta\in W_\rad$ with $(\eta,\eta)_X=1$,
  \[(\Theta_\rad,1)_X=(\eta\star_\rad\eta,1)_X=(\eta,\eta)_X=1.\]
  By \eqref{ci:eq:radial-character} and the Frobenius property,
for $b\in B_\rad$,
\[
\chi_\rad(b)
=(b\star_\rad\eta,\eta)_X
=(b,\Theta_\rad)_X
=(d\Lambda)_\rad(b).
\]
Hence $b\star_\rad\Theta_\rad
=\chi_\rad(b)\Theta_\rad
=(d\Lambda)_\rad(b)\Theta_\rad.$ Choose $\eta'\in W_\rad$ such that
$(\eta',\eta')_X=1$ and $(\eta,\eta')_X=0$.
Then
\[
\chi_\rad(\Theta_\rad)\eta'
=\Theta_\rad\star_\rad\eta'
=\eta\star_\rad\eta\star_\rad\eta'=0,
\]
so $\chi_\rad(\Theta_\rad)=0$.
For $w\in W_\rad$, we have
$w\star_\rad\Theta_\rad=0=(d\Lambda)_\rad(w)$.
This proves \eqref{ci:eq:Theta-x} over $R_\rad$.
  It follows that $\Theta_{\mathrm{rad}}^2=0$.
  This proves \eqref{ci:eq:Theta-identities} over $R_\mathrm{rad}$.
  The tensors
  \[
    (\Theta\star-)-d\Lambda(-)\Theta,\quad
    (\Theta,1)_X-1,\quad\Theta^2
  \]
  are $O(V)$-invariant, so Lemma~\ref{ci:lem:radial-testing} makes them zero over $R$.
  Finally, by homogeneity of the potential,
  \[
    E\Phi=-\Phi+\tfrac12\int_Xc_1(X)\cup\tau\cup\tau;
  \]
  see \cite[Sections~2.1 and~4.3]{HuBigQHCI}.
  The classical term is independent of $u$, since $H\cup V=0$.
  Therefore $\partial_s(E\Phi)=-\Phi_s.$
  On the other hand, Leibniz rule gives $\partial_s(E\Phi)=E(\Phi_s)-2\Phi_s.$
  Hence $E\Lambda=\Lambda$. 
  Substituting $x=E$ in \eqref{ci:eq:Theta-x} gives $E\star\Theta=\Lambda\Theta$. 
  Hence $\chi_\rad(E)=\Lambda_\rad$. 
  By \eqref{ci:eq:radial-character}, $U_\rad|_{W_\rad}=\Lambda_\rad I$.
\end{proof}

\begin{lem}
\label{ci:lem:simple-eigenvalue}
    Suppose $f\in k$ is a simple eigenvalue of $U$.
    Then
    \begin{equation*}
        U\grad f=f\grad f, \qquad (\grad f, 1)_X=1.
    \end{equation*}
\end{lem}
\begin{proof}
    Let $L_f\subset\mathcal H\otimes_R k$ be the $f$-eigenline of $U$.
    Since every multiplication operator commutes with $U$, for each
    $a\in\mathcal H\otimes_R k$ there is a scalar $\chi_f(a)\in k$ such that
    $C_a|_{L_f}=\chi_f(a)I$. Thus
    $\chi_f:\mathcal H\otimes_R k\to k$ is a unital algebra character.

    For $x=\sum_i x^i\phi_i\in H^{\ev}(X)$, set
    $\partial_x=\sum_i x^i\partial_{\tau^i}$. Choose $e\in L_f$ and
    $\ell\in(\mathcal H\otimes_R k)^\vee$ such that
    $Ue=fe$, $\ell U=f\ell$, and $\ell(e)=1$. Differentiating $Ue=fe$ gives
    \[
        \partial_xf=\ell((\partial_xU)e).
    \]
    Flatness of the quantum connection gives
    $\partial_xU=C_x+[C_x,\mu]$. Since $C_x$ commutes with $U$, one has
    $C_xe=\chi_f(x)e$ and $\ell C_x=\chi_f(x)\ell$. Hence
    $\ell([C_x,\mu]e)=0$, and therefore
    \[
        \partial_xf=\chi_f(x).
    \]
    It follows from the definition of the gradient that
    $(\grad f,x)_X=\chi_f(x)$ for every $x\in H^{\ev}(X)$, and hence,
    after extending scalars, for every $x\in\mathcal H\otimes_R k$.

    Put $v=\grad f$. For $a,b\in\mathcal H\otimes_R k$, the Frobenius
    property gives
    \[
        (a\star v,b)_X
        =(v,a\star b)_X
        =\chi_f(a\star b)
        =\chi_f(a)\chi_f(b)
        =(\chi_f(a)v,b)_X.
    \]
    By nondegeneracy, $C_av=\chi_f(a)v$. Taking $a=E$ gives
    $Uv=fv$, and taking $a=1$ gives $(v,1)_X=\chi_f(1)=1$.
\end{proof}

\subsection{Fourfold atoms}

For $X=X_{d_1,\ldots,d_c}\subset\PP(a_0,\ldots,a_{4+c})$, write 
\[
\mathsf D=\frac{\prod_jd_j^{d_j}}{\prod_i a_i^{a_i}},\qquad
\alpha=\begin{cases}
    \frac{\prod_jd_j!}{\prod_i a_i!} & \iota=1\\
    0 & \iota\ge2
\end{cases}.
\]
Let $\Gamma$ be the multiset of rational numbers
\[\Gamma=\bigsqcup_j\{1/d_j,\ldots,d_j/d_j\}-\bigsqcup_i\{1/a_i,\ldots,a_i/a_i\},\]
where the exclusion is in the sense of multisets.
\begin{prop}
\label{ci:prop:amb-sm-qh}
    (1) The ambient small quantum cohomology $(A_\mathrm{sm},\star_\mathrm{sm})$ is cyclically generated by $H$ with the sole relation 
    \[(H+\alpha Q)^{5-\iota}\left(( H+\alpha Q)^\iota-\mathsf DQ\right)=0.\] 
    (2) The multiset of small formal exponents corresponding to the eigenvalue $-\alpha Q$ is 
    \[\{[\iota\gamma]:\gamma\in\Gamma\}.\]
    (3) Small quantum multiplication by $H$ on the primitive part $V$ satisfies \[(H+\alpha Q)\star_\mathrm{sm}V=0.\]
\end{prop}
\begin{proof}
    (1) This follows from the mirror theorem \cite[Theorem~1.1]{PrzyjalkowskiWeightedQH} by standard procedure; see e.g. \cite[Section 6.1]{KKPY}. 
    
    (2) We only explain the case $\iota\ge2$, since this is the only case that will be needed, and the case $\iota=1$ is similar.
    Put $\zeta=Q/z^\iota$, $\vartheta=\zeta\partial_\zeta$, and $p_X(x)=\prod_{\gamma\in\Gamma}(x+\gamma)$. 
The quantum differential equation given in \cite[Theorem~1.1]{PrzyjalkowskiWeightedQH} is
\begin{equation*}
  \label{ci:eq:index-two-GKZ}
  \left(\vartheta^5-\mathsf D\zeta p_X(\vartheta)\right)I_X=0.
\end{equation*}
This is the scalar reduction of the rank-five ambient part of the small quantum connection.
We need to study its formal solutions at $\zeta=\infty$.
Consider a solution
\[
  \zeta^\beta\sum_{\ell\ge0}c_\ell\zeta^{-\ell},
  \qquad c_0\ne0.
\]
The coefficient of the highest power $\zeta^{\beta+1}$ gives
\[
  p_X(\beta)=0.
\]
Thus the exponent classes are
$[\iota\gamma]\in\CC/\mathbb Z$.

(3) This follows from the standard argument using $O(V)$-invariance; see, for example,  \cite[Theorem 1]{CollinoJinzenjiHypersfQH}.
\end{proof}

Write $\lambda=-\iota\alpha Q$.
\begin{cor}
\label{ci:cor:amb-jordan}
    The operator $U_\amb:\mathcal{H}_\amb\to\mathcal{H}_\amb$ has eigenvalues $\Lambda_\amb,\Xi_{1,\amb},\dots,\Xi_{\iota,\amb}\in R_\amb$, and Jordan decomposition 
    \[J_{5-\iota}(\Lambda_\amb)\oplus J_1(\Lambda_\amb)^{\oplus N}
  \oplus\bigoplus_{n=1}^\iota
  J_1(\Xi_{n,\amb}).\]
\end{cor}
\begin{proof}
    By Proposition \ref{ci:prop:amb-sm-qh}, the operator $U_\sm:\mathcal{H}_\sm\to\mathcal{H}_\sm$ has Jordan decomposition 
    \[J_{5-\iota}(\lambda)\oplus J_1(\lambda)^{\oplus N}
  \oplus\bigoplus_{\zeta^\iota=\mathsf DQ}
  J_1(\iota\zeta+\lambda).\] Since $A_\amb$ is cyclically generated by $E_\amb$, Theorem \ref{thm:formal-subscheme-persistence} applies.
\end{proof}

\begin{prop}
\label{ci:prop:atom-constraints}
    The operator $U:\mathcal{H}\to\mathcal{H}$ has eigenvalue $\Lambda\in R_\mathrm{inv}$, with the following properties: 
    \begin{enumerate}
        \item Its multiplicity $m\in[N,N+5-\iota].$
        \item The atom $\mathcal M_\Lambda$ corresponding to $\Lambda$ admits a decomposition 
        \[\mathcal M_\Lambda\simeq\mathcal E^{-\Lambda/z}\otimes
    \bigl(\mathbf1^{\oplus N}\oplus\mathcal R_\Lambda\bigr),\]
    where $\mathbf1=(k\llaur z\rlaur ,\partial_z)$,
$\mathcal E^{-\Lambda/z}=(k\llaur z\rlaur ,\partial_z-\Lambda/z^2)$, and $\mathcal{R}_\Lambda$ is a differential module of rank $m-N$.
        \item If $m=N+5-\iota$, then its Jordan partition is either $(6-\iota,1^{N-1})$ or $(5-\iota,1^N)$.
    \end{enumerate}
\end{prop}
\begin{proof}
    (1) The upper bound follows from Proposition \ref{ci:prop:amb-sm-qh}, and the lower bound follows from the rank-$N$ summand constructed in (2).

    (2) For any constant $v\in V$, put
  \[
    a_v=(u,v)_X,\qquad S_v(z)=v-\frac{a_v}{z}\Theta,
    \qquad Y_v(z)=e^{-\Lambda/z}S_v(z).
  \]
  Twisting by the exponential factor $\mathcal{E}^{\Lambda/z}$ gives
  \[
    \widetilde\nabla_\xi=\partial_\xi+\frac{C_\xi-\partial_\xi\Lambda I}{z},
    \qquad
    \widetilde\nabla_z=\partial_z-\frac{U-\Lambda I}{z^2}+\frac\mu z,
  \] where $C_\xi=\xi\star-$ and $\mu$ is the grading operator.
  We first show that $\widetilde\nabla_\xi S_v=\widetilde\nabla_zS_v=0$.
  By \eqref{ci:eq:primitive-chain-rule},
  \begin{equation}
    \label{ci:eq:primitive-product}
    C_\xi v=\partial_\xi\Lambda v+(\partial_\xi a_v)\Theta
    +a_v\partial_\xi\Theta.
  \end{equation}
  Together with $(C_\xi-\partial_\xi\Lambda I)\Theta=0$ from
  \eqref{ci:eq:Theta-x}, this yields
  \[
    \widetilde\nabla_\xi S_v
    =\frac{(C_\xi-\partial_\xi\Lambda I)v-(\partial_\xi a_v)\Theta
    -a_v\partial_\xi\Theta}{z}=0.
  \]
  By flatness $[\nabla_\xi,\nabla_z]=0$, we obtain
  $\partial_\xi U=C_\xi+[C_\xi,\mu]$. Applying this identity to the unit $1$, and using $U(1)=E$ and
$\mu(1)=(-2)\,1$, we obtain, for every $\xi$,
\[
dE(\xi)=\partial_\xi E
=(C_\xi+[C_\xi,\mu])(1)
=-\xi-\mu\xi.
\]
Thus, $dE=-I-\mu.$ Since $\mu^*=-\mu$ with respect to the Poincar\'e pairing, it follows that $(dE)^*=-I+\mu.$ Taking the gradient of $E\Lambda=\Lambda$ gives $
\partial_E\Theta+(dE)^*\Theta=\Theta.$ Hence $\partial_E\Theta=(2I-\mu)\Theta$.
  Substituting $\xi=E$ in \eqref{ci:eq:primitive-product} and using
  $E(a_v)=-a_v$ gives $ (U-\Lambda I)v=a_v(I-\mu)\Theta.$ Since $\mu v=0$ and $(U-\Lambda I)\Theta=0$,
  \[
    \widetilde\nabla_zS_v
    =\frac{a_v\Theta-(U-\Lambda I)v-a_v\mu\Theta}{z^2}=0.
  \]
  
  We now prove the decomposition.
  Equations~\eqref{ci:eq:Theta-gradient} and \eqref{ci:eq:Theta-identities}
  give $(\Theta,v)_X=\Phi_{ss}a_v$ and $(\Theta,\Theta)_X=0$, hence
  \begin{equation}
    \label{ci:eq:primitive-pairing}
    (S_v(-z),S_w(z))_X=(v,w)_X.
  \end{equation}
  Hence, for an orthonormal basis $\phi_1,\ldots,\phi_N$ of $V$,
  \[
    \mathcal P(z)(a)
    =
    \sum_{i=1}^N
    S_{\phi_i}(z)(S_{\phi_i}(-z),a)_X
  \]
  is an idempotent with image
  $L_R:=\sum_{i=1}^N\mathcal A_zS_{\phi_i}$. Since the $S_{\phi_i}$
  are horizontal and the quantum pairing is flat,
  $\mathcal P$ is horizontal. Thus
  \[
    \mathcal H\otimes_R\mathcal A_z
    =
    L_R\oplus\Ker\mathcal P
  \]
  is a decomposition into differential submodules for the twisted
  connection, and $L_R$ is a trivial differential module of rank $N$. 

After extension to $k\llaur z\rlaur$, this summand lies in the
  twisted $\Lambda$-atom:
  \[
    L_R\otimes_{\mathcal A_z}k\llaur z\rlaur
    \subset\mathcal E^{\Lambda/z}\otimes\mathcal M_\Lambda.
  \]
  Indeed, use the formal block-diagonal frame of
  Proposition~\ref{prop:euler-spectral-block}. If an $S_v$ had a
  nonzero component $a(z)=z^qa_q+\cdots$, $a_q\ne0$, in a block
  with leading eigenvalue $\Gamma\ne\Lambda$, its horizontality
  would give, at order $z^{q-2}$,
  \[
    (U-\Lambda I)a_q=0.
  \]
  This is impossible because $a_q$ belongs to the generalized
  $\Gamma$-eigenspace, on which $U-\Lambda I$ is invertible.
  Restricting $\mathcal P$ to the twisted $\Lambda$-atom therefore gives
  \[
    \mathcal E^{\Lambda/z}\otimes\mathcal M_\Lambda
    \simeq\mathbf1^{\oplus N}\oplus\mathcal R_\Lambda,
  \]
  where $\mathcal R_\Lambda$ is the kernel of this restriction and
  has rank $m-N$.

(3) Since $\Lambda_\sm=\lambda$ and $m=N+5-\iota$, the Hensel
lift of the small factor $(x-\lambda)^m$ is $(x-\Lambda)^m$.
Its coprimality with the complementary factor gives a spectral
projector $P_\Lambda\in\End_R(\mathcal H)$, which is
$O(V)$-equivariant by uniqueness. Put
$\mathcal N=(U-\Lambda I)P_\Lambda$.

Applying Lemma~\ref{ci:lem:radial-testing} to the invariant tensors
$\bigwedge^j(\mathcal N^a)$ shows that every power of $\mathcal N$
has the same rank over $k$ as its radial restriction over $k_\rad$.
Since these ranks determine the Jordan partition, it suffices to
work over $k_\rad$.

The decomposition $\mathcal H_\rad=B_\rad\oplus W_\rad$ is
preserved by $U_\rad$, and
$U_\rad|_{W_\rad}=\Lambda_\rad I$ by
\eqref{ci:eq:Theta-eigen}. Thus $W_\rad$ contributes $N-1$
size-one blocks, leaving a $\Lambda_\rad$-generalized eigenspace
of dimension $6-\iota$ in $B_\rad$ after extending scalars.
By Proposition~\ref{ci:prop:amb-sm-qh},
$\mathcal N_\sm^{\,4-\iota}\ne0$, hence
$\mathcal N_\rad^{\,4-\iota}\ne0$.
The remaining eigenspace therefore contains a Jordan block of
size at least $5-\iota$, so the full partition is
\[
  (6-\iota,1^{N-1})\qquad\text{or}\qquad(5-\iota,1^N).
\]
\end{proof}

\subsection{Index three case}\label{subsec:index-three}

\begin{thm}
  \label{ci:thm:index-three}
  $X_6\subset\PP(1^4,2,3)$ and
  $X_4\subset\PP(1^5,2)$ are symplectically irrational.
\end{thm}
\begin{proof}
  Assume, for contradiction, $X$ is symplectically rational. By Proposition \ref{ci:prop:atom-constraints}(1), $X$ has an atom with $m\ge N$. 
  Since $N>24$ in both families, by Proposition \ref{prop:surface-atom-list}, this atom must come from the atom of a surface of Kodaira dimension 1 or 2. 
  The number of trivial summands $\mathbf1$'s of a surface atom is $m-4$ if $\kappa=1$ and $m-3$ if $\kappa=2$. 
  By Proposition \ref{ci:prop:atom-constraints}(2), the atom of $X$ has at least $N$ trivial summands, so $m\ge N+3$. Since $\iota=3$, this contradicts Proposition \ref{ci:prop:atom-constraints}(1).
\end{proof}

\begin{rem}
  \label{ci:rem:cubic-fourfold}
  For $X_3\subset\PP^5$, $N=22$. 
  The same argument shows that if there is a special rational member, then its weak factorization must involve a blow-up along a K3 surface. 
  Special rational members do in fact exist \cite{BolognesiRussoStaglianoRationalCubics}.
  On the other hand, the very general member is irrational \cite{KKPY}; the proof uses atoms enhanced with Hodge theory.
\end{rem}

\begin{rem}
    \label{ci:rem:small-monodromy}
    For $X_{2,2}\subset\PP^6$, the Zariski closure of its monodromy representation is the standard representation of the Weyl group of $D_7$ \cite{SGA7II,DeligneWeilII}.
    This group is too small to sufficiently constrain Gromov--Witten theory, and for this reason our strategy does not apply.
\end{rem}

\subsection{Index two case}\label{subsec:index-two}

\begin{thm}
  \label{ci:thm:index-two}
  $X_6\subset\PP(1^5,3)$, $X_4\subset\PP^5$, and
  $X_{2,3}\subset\PP^6$ are symplectically irrational.
\end{thm}
\begin{proof}
{\it Step 1. Non-splitting of the eigenvalue.}
    Assume, for contradiction, $X$ is symplectically rational. By Proposition \ref{ci:prop:atom-constraints}(1), $X$ has an atom with $m\ge N$. 
  Since $N>24$ in all three examples, by Proposition \ref{prop:surface-atom-list}, this atom must come from the atom of a surface of Kodaira dimension 1 or 2. 
  The number of trivial summands $\mathbf1$'s of a surface atom is $m-4$ if $\kappa=1$ and $m-3$ if $\kappa=2$. 
  By Proposition \ref{ci:prop:atom-constraints}(2), the atom of $X$ has at least $N$ trivial summands, so $m\ge N+3$. Since $\iota=2$, by Proposition \ref{ci:prop:atom-constraints}(1), $m=N+3$, and the atom is that of a surface of $\kappa=2.$

  {\it Step 2. Comparing the formal exponents.}
  By \eqref{eq:surface-connection-types}, the connection decomposes as 
  \[\mathcal M_\Lambda
    \cong
    \mathcal E^{-\Lambda/z}\otimes
    \bigl(\mathbf1^{\oplus N}\oplus\mathcal U_3\bigr).\]
 It is of unramified exponential type, so we are in the situation of Proposition~\ref{lem:non-split-unramified}. All formal exponents are 0.
 On the other hand, by Proposition \ref{ci:prop:amb-sm-qh}, the small formal exponents corresponding to the 0 eigenvalue are computed as follows: 
  \[
  \begin{array}{c|c|c}
    X&m=N+3&\Exp^{\sm}_0\\
    \hline
    X_6\subset\PP(1^5,3)&524&\{0^{522},1/3,2/3\}\\
    X_4\subset\PP^5&186&\{0^{184},(1/2)^2\}\\
    X_{2,3}&88&\{0^{86},1/3,2/3\}
  \end{array}.
\] This gives a contradiction.
\end{proof}

\begin{rem}
  \label{ci:rem:three-quadrics}
  For $X_{2,2,2}\subset\PP^7$, the small exponent multiset is $\{0^{46}\}$. 
  The first half of the argument still applies, and shows that if there is a special rational member, then its weak factorization must involve a blow-up along a surface of general type with $b_2=44$.

  Special rational members do in fact exist, whereas the very general member is irrational
  \cite[Theorem~1]{HassettPirutkaTschinkelThreeQuadrics}.
\end{rem}

The following proposition shows that the eigenvalue does not split in the index two case. It is not needed for our main results, but we include it since it may be of independent interest.

\begin{prop}
\label{ci:prop:exact-index-two}
For each of the four index-two families in
Table~\ref{ci:tab:irrationality}, there exists
$\mathsf R(r)\in R_r^\times$, with $\mathsf R(0)=4\mathsf DQ$, such that
\begin{equation}
\label{ci:eq:exact-index-two}
U_r^3\bigl(U_r^2-\mathsf R(r)I\bigr)=0,\qquad
\det(xI-U_r)=x^{N+3}\bigl(x^2-\mathsf R(r)\bigr).
\end{equation}
Moreover, $m_\Lambda=N+3$ and $\mathcal J_\Lambda=(3,1^N)$ over $k$.
In particular, writing $\Xi_1,\Xi_2$ for the two simple eigenvalues,
\begin{equation}
\label{ci:eq:big-index-two-relation}
(U-\Lambda I)^3(U-\Xi_1I)(U-\Xi_2I)=0.
\end{equation}
\end{prop}

\begin{proof}
Work over $R_r=K\llbracket r\rrbracket$, where $t=0$, $u=r\phi$, and
$E_r=2H-r\phi$. For $j\ge1$, a genus-zero invariant of degree
$d=H\cdot\beta$ with $2j$ primitive insertions can be nonzero only if
\[
4j=\vdim_\CC\overline M_{0,2j}(X,\beta)=1+2d+2j.
\]
This would give $2d=2j-1$, which is impossible; for $d=0$, the
two-point term is absent from the potential. Hence
$\Lambda_\pr=(\Phi_s)_\pr=0$. By
Lemma~\ref{ci:lem:Theta-identities},
\[
E_r\star_r\Theta_r=0,\qquad U_r|_{W_r}=0.
\]

We next express $\Theta_r$ in terms of Euler powers. The reflection
$\phi\mapsto-\phi$, together with the quantum grading, gives
\[
E_r^2,E_r^4,\Theta_r
\in R_r1\oplus R_rH_2\oplus R_rH_4.
\]
Indeed, the ambient coefficients are even in $r$, while the
$\phi$-coefficients are odd; the grading excludes the latter terms.
Since $1,E_r^2,E_r^4$ specialize at $r=0$ to a basis of this
three-dimensional submodule by
Proposition~\ref{ci:prop:amb-sm-qh}(1), they form an $R_r$-basis.
Thus
\[
\Theta_r=a_0\,1+a_2E_r^2+a_4E_r^4,
\qquad a_0,a_2,a_4\in R_r.
\]
By Lemma~\ref{ci:lem:Theta-identities}, the character $\chi_\rad$
vanishes on both $E_r$ and $\Theta_r$. Applying it to the preceding
identity gives $a_0=0$.

At the small point, Proposition~\ref{ci:prop:amb-sm-qh}(1) shows that
$\ker(U_\sm|_{A_\sm})$ is generated by
$E_\sm^4-4\mathsf DQ E_\sm^2$. Its pairing with $1$ is $16d_X$.
Since $U_\sm\Theta_\sm=0$ and $(\Theta_\sm,1)_X=1$,
\[
\Theta_\sm
=\frac{E_\sm^4-4\mathsf DQ E_\sm^2}{16d_X}.
\]
Therefore
$a_4(0)=1/(16d_X)$ and
$a_2(0)=-\mathsf DQ/(4d_X)$, so $a_4$ is a unit. Set
$\mathsf R(r)=-a_2/a_4$. Then $\mathsf R(0)=4\mathsf DQ$ and $\Theta_r=a_4E_r^2\bigl(E_r^2-\mathsf R(r)1\bigr)$. Since $E_r\star_r\Theta_r=0$, we obtain $E_r^3\bigl(E_r^2-\mathsf R(r)1\bigr)=0$, 
and hence, by associativity,
\[
U_r^3\bigl(U_r^2-\mathsf R(r)I\bigr)=0.
\]

Choose $\sqrt{\mathsf R(r)}\in R_r^\times$. Since
$x^3$, $x-\sqrt{\mathsf R(r)}$, and $x+\sqrt{\mathsf R(r)}$
are pairwise coprime, they give a decomposition of the active part
into summands of ranks $4,1,1$, as seen at $r=0$. On the rank-four
summand, $U_r^3=0$, while $U_r^2\ne0$ because its specialization
contains a $J_3(0)$-block. Hence its Jordan partition is $(3,1)$.
Together with $W_r$, on which $U_r=0$, the zero eigenvalue has
partition $(3,1^N)$, while the other two eigenvalues are simple.
Thus
\[
\det(xI-U_r)=x^{N+3}\bigl(x^2-\mathsf R(r)\bigr).
\]

By $O(V)$-equivariance, the same multiplicities and Jordan partitions
hold over $k_\pr$: after extending coefficients, a non-isotropic
primitive vector can be carried to $r\phi$ by an element of $O(V)$.
Moreover, since $TS_\pr=V_\pr$, Proposition~\ref{ci:prop:amb-sm-qh}(1)
gives
\[
TS_\pr+\Span_{R_\pr}
\{1,E_\pr,E_\pr^2,E_\pr^3,E_\pr^4\}
=\mathcal H_\pr,
\]
because the images of these five Euler powers modulo $V_\pr$ form a
basis at $u=0$. Theorem~\ref{thm:formal-subscheme-persistence}
therefore transports the same multiplicity and Jordan partition to
the full big quantum cohomology.

Finally, the repeated eigenvalue is $\Lambda$: one has
$U\Theta=\Lambda\Theta$ and $(\Theta,\Theta)_X=0$, whereas a simple
eigenspace of the self-adjoint operator $U$ is nondegenerate.
Consequently
\[
m_\Lambda=N+3,\qquad
\mathcal J_\Lambda=(3,1^N),
\]
and \eqref{ci:eq:big-index-two-relation} follows.
\end{proof}

\subsection{Index one case}\label{subsec:index-one}

\begin{thm}
  \label{ci:thm:index-one}
  The index-one families Nos.~1--6 and 8--13
  in Table~\ref{ci:tab:irrationality} are symplectically irrational.
\end{thm}
\begin{proof}
    {\it Step 1. Splitting of the eigenvalue.}
    Assume, for contradiction, $X$ is symplectically rational. By Proposition \ref{ci:prop:atom-constraints}(1), $X$ has an atom with $m\ge N$. 
  Since $N>24$ for all index-one families, by Proposition \ref{prop:surface-atom-list}, this atom must come from the atom of a surface of Kodaira dimension 1 or 2. 
  The number of trivial summands $\mathbf1$'s of a surface atom is $m-4$ if $\kappa=1$ and $m-3$ if $\kappa=2$. 
  By Proposition \ref{ci:prop:atom-constraints}(2), the atom of $X$ has at least $N$ trivial summands, so $m\ge N+3$. 
  Since $\iota=1$, by Proposition \ref{ci:prop:atom-constraints}(1), $m=N+3$ or $m=N+4$. If $m=N+4$, then by Proposition \ref{ci:prop:atom-constraints}(3), the Jordan partition is $(5,1^{N-1})$ or $(4,1^{N})$, which is impossible, since the atoms of surfaces of $\kappa>0$ have Jordan blocks of size at most 3 by Proposition \ref{prop:surface-atom-list}.
  Hence $m=N+3$, and the Jordan partition is $(3,1^N)$. In other words, the small eigenvalue $\lambda$ of multiplicity $N+4$ splits, under bulk deformation, into an eigenvalue $\Lambda$ of multiplicity $N+3$ and an eigenvalue $\Lambda'$ of multiplicity $1$.

  {\it Step 2. Excluding the $(\Lambda,\Lambda')$-splitting.} 
We will show that the splitting into $\Lambda$ and $\Lambda'$ is impossible.
By Corollary~\ref{ci:cor:amb-jordan},
\[
\det(xI-U_\amb)
=(x-\Lambda_\amb)^{N+4}(x-\Xi_\amb).
\]
The two factors are coprime over $R_\amb$. Hence Hensel lifting gives
coprime monic polynomials $q_0,q_1\in R_\inv[x]$ such that $\det(xI-U)=q_0(x)q_1(x)$ with 
\[
(q_0)_\amb=(x-\Lambda_\amb)^{N+4},
\qquad
(q_1)_\amb=x-\Xi_\amb.
\]
Choose $p_0,p_1\in R_\inv[x]$ such that $p_0q_0+p_1q_1=1$, and set
\[
P_0=p_1(U)q_1(U),\qquad
P_1=p_0(U)q_0(U).
\]
Then $P_i\in\End_{R}(\mathcal H)$ are the corresponding $O(V)$-equivariant
spectral projectors. Set $\mathcal{H}_i=\mathrm{im}P_i$.

Over $k$, Step~1 identifies
\[
q_0(x)=(x-\Lambda)^{N+3}(x-\Lambda'),
\qquad
q_1(x)=x-\Xi.
\]
Therefore, $\Lambda'=\tr(P_0U)-(N+3)\Lambda.$ Since $\Lambda\in R_\inv$, it follows that $\Lambda'\in R_\inv$. Put
$\Delta=\Lambda'-\Lambda$. Then
\[
\Delta=\tr(P_0U)-(N+4)\Lambda\in R_\inv.
\]
Restricting this identity to the ambient locus gives
$\Delta_\amb=\Delta(t,0)=0$. 

Consider the operator 
\[
\mathcal T=(U-\Lambda I)^3P_0\in\End_R(\mathcal H).
\]
Let $L_\Lambda$ denote the generalized $\Lambda$-eigenspace over $k$, and let $L_{\Lambda'}$ denote $\Lambda'$-eigenline.
Since the Jordan partition of $\Lambda$ is $(3,1^N)$,
\[
\mathcal T|_{L_\Lambda}=0,\qquad
\mathcal T|_{L_{\Lambda'}}=\Delta^3I,\qquad
\mathcal T|_{\mathcal H_1}=0.
\]
Hence $\rk\mathcal T=1$, $\im\mathcal T=L_{\Lambda'}$, and
$\tr\mathcal T=\Delta^3$ over $k$.

Let $v=\grad\Lambda'$. By Lemma~\ref{ci:lem:simple-eigenvalue},
$Uv=\Lambda'v$ and $(v,1)_X=1$. Since $\im\mathcal T=kv$, write
$\mathcal T\xi=c(\xi)v$. Self-adjointness gives
\[
c(\xi)=(\mathcal T\xi,1)_X
=(\xi,\mathcal T1)_X
=c(1)(\xi,v)_X.
\]
Set $b=c(1)=(\mathcal T1,1)_X$. Then
\[
\mathcal T\xi=b(\xi,v)_Xv,
\qquad
\Delta^3=\tr\mathcal T=b(v,v)_X.
\]

On the other hand, $U\Theta=\Lambda\Theta$ and $Uv=\Lambda'v$.
Since $U$ is self-adjoint, $\Delta(v,\Theta)_X=0$, and hence
$(v,\Theta)_X=0$. Also $\Theta\star\Theta=0$ implies
$(\Theta,\Theta)_X=0$; see~\eqref{ci:eq:Theta-identities}. Since $v-\Theta=\grad\Delta$, we obtain
\[
(v,v)_X=(\grad\Delta,\grad\Delta)_X.
\]
Therefore
\begin{equation}
\label{ci:eq:cubic-gap}
\Delta^3=b(\grad\Delta,\grad\Delta)_X.
\end{equation}

We next show that $b$ is a unit. At the small point, put
$h=H+\alpha Q\,1$. By Proposition~\ref{ci:prop:amb-sm-qh},
$A_\sm=K[h]/(h^4(h-\mathsf DQ))$, and the idempotent corresponding
to the repeated factor is
\[
e_0=1-\frac{h^4}{(\mathsf DQ)^4}.
\]
Since $U_\sm-\Lambda_\sm I=C_h$ and $P_{0,\sm}(1)=e_0$,
\[
\mathcal T_\sm1
=h^3e_0
=h^3-\frac{h^4}{\mathsf DQ}
=-\frac{h^4-\mathsf DQh^3}{\mathsf DQ}
=-\frac{d_X}{\mathsf DQ}\Theta_\sm.
\]
Using $(\Theta_\sm,1)_X=1$ in ~\eqref{ci:eq:Theta-identities}, we obtain
\begin{equation}
\label{ci:eq:gap-unit}
b_\sm=-\frac{d_X}{\mathsf DQ}\ne0.
\end{equation}
Hence $b\in R_\inv^\times$. By \eqref{ci:eq:invariant-gradient}, we have the following identity
\begin{equation}
\label{ci:eq:coordinate-gap}
\Delta^3
=b\left(
\sum_{a,c=0}^4g^{ac}\Delta_a\Delta_c+2s\Delta_s^2
\right).
\end{equation}
Suppose $\Delta\ne0$. Then 
$\Delta\in R_\amb\llbracket s\rrbracket$ has positive $s$-order, so that we can write 
$\Delta=c(t)s^q+O(s^{q+1})$, where $q\ge1$ and
$0\ne c(t)\in R_\amb$. Then
\[
\sum_{a,c=0}^4g^{ac}\Delta_a\Delta_c=O(s^{2q}),
\qquad
2s\Delta_s^2
=2q^2c(t)^2s^{2q-1}+O(s^{2q}).
\]
Since $b\in R_\inv^\times$, the two sides of
\eqref{ci:eq:coordinate-gap} have $s$-orders $3q$ and $2q-1$,
respectively, which is impossible for $q\ge1$. Thus $\Delta=0$,
contradicting $\Lambda'\ne\Lambda$. Therefore the splitting cannot
occur, and $X$ is symplectically irrational.
\end{proof}

\begin{rem}
  \label{ci:rem:weighted}
  \label{ci:rem:double-quadric}
  The family $X_{2,6}\subset\PP(1^6,3)$ consists of double covers of smooth quadric fourfolds branched over a smooth divisor of degree 6. 
  All members are known to be irrational \cite[Section~0]{PukhlikovDoubleQuadric}.
  The primitive cohomology of the quadric pulls back to a
nondegenerate monodromy-invariant line $L\subset V$. Thus the monodromy closure is contained in $O(L)\times O(L^\perp)\subsetneq O(V)$, and the one-variable monodromy reduction used above does not apply directly.
  It is plausible that our arguments can be generalized to deal with this example, but we do not attempt it in this paper.
\end{rem}

\printbibliography

@article{AKMWWeakFactorization,
  author = {Abramovich, Dan and Karu, Kalle and Matsuki, Kenji and W{\l}odarczyk, Jaros{\l}aw},
  title = {Torification and factorization of birational maps},
  journaltitle = {Journal of the American Mathematical Society},
  volume = {15},
  number = {3},
  pages = {531--572},
  year = {2002},
  doi = {10.1090/S0894-0347-02-00396-X},
  eprint = {math/9904135},
  eprinttype = {arxiv},
  url = {https://arxiv.org/abs/math/9904135},
}

@misc{AudibertDoubaLeeMarquis,
  author = {Audibert, Jacques and Douba, Sami and Lee, Gye-Seon and Marquis, Ludovic},
  title = {{Zariski}-closures of linear reflection groups},
  year = {2025},
  eprint = {2504.01494v1},
  eprinttype = {arxiv},
  url = {https://arxiv.org/abs/2504.01494v1},
}

@article{BolognesiRussoStaglianoRationalCubics,
  author = {Bolognesi, Michele and Russo, Francesco and Staglian{\`o}, Giovanni},
  title = {Some loci of rational cubic fourfolds},
  journaltitle = {Mathematische Annalen},
  volume = {373},
  number = {1--2},
  pages = {165--190},
  year = {2019},
  eprint = {1504.05863},
  eprinttype = {arxiv},
  url = {https://arxiv.org/abs/1504.05863},
}

@misc{CaiCubicSymplecticIrrationality,
  author = {Cai, Jiaji},
  title = {The cubic threefold is symplectically irrational},
  year = {2026},
  eprint = {2608.01577v1},
  eprinttype = {arxiv},
  url = {https://arxiv.org/abs/2608.01577v1},
}

@misc{CaiQuarticSymplecticIrrationality,
  author = {Cai, Jiaji},
  title = {The quartic threefold is symplectically irrational},
  year = {2026},
  eprint = {2605.29143v2},
  eprinttype = {arxiv},
  url = {https://arxiv.org/abs/2605.29143v2},
}

@article{DiGennaroFranco,
  author = {Di Gennaro, Vincenzo and Franco, Davide},
  title = {Monodromy of a family of hypersurfaces},
  journaltitle = {Annales scientifiques de l'{\'E}cole Normale Sup{\'e}rieure},
  series = {4},
  volume = {42},
  number = {3},
  pages = {517--529},
  year = {2009},
  doi = {10.24033/asens.2101},
  eprint = {0803.1627},
  eprinttype = {arxiv},
  url = {https://www.numdam.org/articles/10.24033/asens.2101/},
}

@article{DimcaWeightedMonodromy,
  author = {Dimca, Alexandru},
  title = {Monodromy and {Betti} numbers of weighted complete intersections},
  journaltitle = {Topology},
  volume = {24},
  number = {3},
  pages = {369--374},
  year = {1985},
  url = {https://imar.ro/~increst/1984/48_1984.pdf},
}

@article{GyengeSzaboBlowups,
  author = {Gyenge, {\'A}d{\'a}m and Szab{\'o}, Szil{\'a}rd},
  title = {Blow-ups and the quantum spectrum of surfaces},
    journaltitle = {Advances in Mathematics},
  volume = {479},
  eid = {110432},
  year = {2025},
  doi = {10.1016/j.aim.2025.110432},
}

@article{HYYZZ,
  author = {Hinault, Thorgal and Yu, Tony Yue and Zhang, Chi and Zhang, Shaowu},
  title = {Decomposition and framing of {$F$}-bundles and applications to quantum cohomology},
  journaltitle = {Proceedings of the London Mathematical Society},
  volume = {133},
  number = {3},
  eid = {e70209},
  year = {2026},
  doi = {10.1112/plms.70209},
  eprint = {2411.02266v2},
  eprinttype = {arxiv},
  url = {https://arxiv.org/abs/2411.02266v2},
}

@incollection{HassettPirutkaTschinkelThreeQuadrics,
  author = {Hassett, Brendan and Pirutka, Alena and Tschinkel, Yuri},
  title = {Intersections of three quadrics in {$\mathbb P^7$}},
  booktitle = {Surveys in Differential Geometry 2017. Celebrating the 50th Anniversary of the Journal of Differential Geometry},
  series = {Surveys in Differential Geometry},
  volume = {22},
  pages = {259--274},
  publisher = {International Press},
  location = {Somerville, MA},
  year = {2018},
  eprint = {1706.01371},
  eprinttype = {arxiv},
  url = {https://arxiv.org/abs/1706.01371},
}

@misc{HuBigQHCI,
  author = {Hu, Xiaowen},
  title = {Big quantum cohomology of {Fano} complete intersections},
  year = {2025},
  eprint = {1501.03683v6},
  eprinttype = {arxiv},
  url = {https://arxiv.org/abs/1501.03683v6},
}

@misc{IritaniBlowups,
  author = {Iritani, Hiroshi},
  title = {Quantum cohomology of blowups},
  year = {2025},
  eprint = {2307.13555v3},
  eprinttype = {arxiv},
  url = {https://arxiv.org/abs/2307.13555v3},
}

@misc{IritaniKotoProjectiveBundles,
  author = {Iritani, Hiroshi and Koto, Yuki},
  title = {Quantum cohomology of projective bundles},
  year = {2026},
  eprint = {2307.03696v4},
  eprinttype = {arxiv},
  url = {https://arxiv.org/abs/2307.03696v4},
}

@misc{KKPY,
  author = {Katzarkov, Ludmil and Kontsevich, Maxim and Pantev, Tony and Yu, Tony Yue},
  title = {Birational invariants from {Hodge} structures and quantum multiplication},
  year = {2026},
  eprint = {2508.05105v2},
  eprinttype = {arxiv},
  url = {https://arxiv.org/abs/2508.05105v2},
}

@article{KedlayaGoodFormalI,
  author = {Kedlaya, Kiran S.},
  title = {Good formal structures on flat meromorphic connections, {I}: Surfaces},
  journaltitle = {Duke Mathematical Journal},
  volume = {154},
  number = {2},
  pages = {343--418},
  year = {2010},
  doi = {10.1215/00127094-2010-041},
  eprint = {0811.0190v4},
  eprinttype = {arxiv},
  url = {https://arxiv.org/abs/0811.0190v4},
}

@article{KontsevichManinGW,
  author = {Kontsevich, Maxim and Manin, Yuri I.},
  title = {{Gromov--Witten} classes, quantum cohomology, and enumerative geometry},
  journaltitle = {Communications in Mathematical Physics},
  volume = {164},
  number = {3},
  pages = {525--562},
  year = {1994},
  eprint = {hep-th/9402147},
  eprinttype = {arxiv},
  url = {https://arxiv.org/abs/hep-th/9402147},
}

@book{ManinFrobenius,
  author = {Manin, Yuri I.},
  title = {Frobenius Manifolds, Quantum Cohomology, and Moduli Spaces},
  series = {American Mathematical Society Colloquium Publications},
  volume = {47},
  publisher = {American Mathematical Society},
  location = {Providence, RI},
  year = {1999},
  doi = {10.1090/coll/047},
}

@book{MilnorSingularPoints,
  author = {Milnor, John},
  title = {Singular Points of Complex Hypersurfaces},
  series = {Annals of Mathematics Studies},
  volume = {61},
  publisher = {Princeton University Press},
  location = {Princeton, NJ},
  year = {1968},
  doi = {10.1515/9781400881819},
}

@article{PrzyjalkowskiShramovBounds,
  author = {Przyjalkowski, Victor and Shramov, Constantin},
  title = {Bounds for smooth {Fano} weighted complete intersections},
  journaltitle = {Communications in Number Theory and Physics},
  volume = {14},
  number = {3},
  pages = {511--553},
  year = {2020},
  doi = {10.4310/CNTP.2020.v14.n3.a3},
  eprint = {1611.09556},
  eprinttype = {arxiv},
  url = {https://arxiv.org/abs/1611.09556},
}

@article{PrzyjalkowskiWeightedQH,
  author = {Przyjalkowski, Victor},
  title = {Quantum cohomology of smooth complete intersections in weighted projective spaces and singular toric varieties},
  journaltitle = {Sbornik: Mathematics},
  volume = {198},
  number = {9},
  pages = {1325--1340},
  year = {2007},
  eprint = {math/0507232v3},
  eprinttype = {arxiv},
  note = {Theorem numbering follows arXiv:math/0507232v3},
  url = {https://arxiv.org/abs/math/0507232v3},
}

@article{PukhlikovDoubleQuadric,
  author = {Pukhlikov, Aleksandr V.},
  title = {Birational automorphisms of a double space and double quadric},
  journaltitle = {Mathematics of the USSR-Izvestiya},
  volume = {32},
  number = {1},
  pages = {233--243},
  year = {1989},
  doi = {10.1070/IM1989v032n01ABEH000756},
  url = {https://www.mathnet.ru/eng/im1177},
}

@book{GreuelLossenShustin,
  author = {Greuel, Gert-Martin and Lossen, Christoph and Shustin, Eugenii},
  title = {Introduction to Singularities and Deformations},
  series = {Springer Monographs in Mathematics},
  publisher = {Springer},
  location = {Berlin},
  year = {2007},
  url = {https://webhomes.maths.ed.ac.uk/~v1ranick/papers/greuel.pdf},
}

@article{LiSympSurf,
  author = {Li, Tian-Jun},
  title = {Symplectic 4-manifolds with {K}odaira dimension zero},
  journaltitle = {Journal of Differential Geometry},
  volume = {74},
  number = {2},
  pages = {321--352},
  year = {2006},
  url = {http://projecteuclid.org/euclid.jdg/1175266207},
  issn = {0022-040X,1945-743X},
  mrnumber = {2259057},
}

@article{ZhangConeThm,
  author = {Zhang, Weiyi},
  title = {The curve cone of almost complex 4-manifolds},
  journaltitle = {Proceedings of the London Mathematical Society},
  volume = {115},
  number = {6},
  pages = {1227--1275},
  year = {2017},
  doi = {10.1112/plms.12062},
  url = {https://doi.org/10.1112/plms.12062},
  issn = {0024-6115,1460-244X},
  mrnumber = {3741851},
}

@article{LiBettiNum,
  author = {Li, Tian-Jun},
  title = {Quaternionic bundles and {B}etti numbers of symplectic 4-manifolds with {K}odaira dimension zero},
  journaltitle = {International Mathematics Research Notices},
  eid = {37385},
  pagetotal = {28},
  year = {2006},
  doi = {10.1155/IMRN/2006/37385},
  url = {https://doi.org/10.1155/IMRN/2006/37385},
  issn = {1073-7928,1687-0247},
  mrnumber = {2264722},
}

@article{GompfSympSum,
  author = {Gompf, Robert E.},
  title = {A new construction of symplectic manifolds},
  journaltitle = {Annals of Mathematics},
  volume = {142},
  number = {3},
  pages = {527--595},
  year = {1995},
  doi = {10.2307/2118554},
  url = {https://doi.org/10.2307/2118554},
  issn = {0003-486X,1939-8980},
  mrnumber = {1356781},
}

@book{SGA7II,
  title = {Groupes de monodromie en g\'eom\'etrie alg\'ebrique. {II}},
  series = {Lecture Notes in Mathematics},
  volume = {340},
  pages = {x+438},
  publisher = {Springer-Verlag},
  location = {Berlin and New York},
  year = {1973},
  note = {S\'eminaire de G\'eom\'etrie Alg\'ebrique du Bois-Marie 1967--1969 (SGA 7 II), Dirig\'e{} par P. Deligne et N. Katz},
  mrnumber = {354657},
}

@article{DeligneWeilII,
  author = {Deligne, Pierre},
  title = {La conjecture de {W}eil. {II}},
  journaltitle = {Institut des Hautes \'Etudes Scientifiques. Publications Math\'ematiques},
  number = {52},
  pages = {137--252},
  year = {1980},
  url = {http://www.numdam.org/item?id=PMIHES_1980__52__137_0},
  issn = {0073-8301,1618-1913},
  mrnumber = {601520},
}

@article{CollinoJinzenjiHypersfQH,
  author = {Collino, Alberto and Jinzenji, Masao},
  title = {On the structure of the small quantum cohomology rings of projective hypersurfaces},
  journaltitle = {Communications in Mathematical Physics},
  volume = {206},
  number = {1},
  pages = {157--183},
  year = {1999},
  doi = {10.1007/s002200050701},
  url = {https://doi.org/10.1007/s002200050701},
  issn = {0010-3616,1432-0916},
  mrnumber = {1736987},
}

@article{ChenExponentialType,
  author = {Chen, Zihong},
  title = {On the exponential type conjecture},
  journaltitle = {Journal of the American Mathematical Society},
  year = {2026},
  pubstate = {forthcoming},
  eprint = {2409.03922v2},
  eprinttype = {arxiv},
  primaryclass = {math.SG},
  url = {https://arxiv.org/abs/2409.03922v2},
}

@misc{DonagiPantevThreeQuadrics,
  author = {Donagi, Ron and Pantev, Tony},
  title = {The intersection of three quadrics in {$\mathbb P^7$} revisited},
  year = {2025},
  eprint = {2503.19039v1},
  eprinttype = {arxiv},
  url = {https://arxiv.org/abs/2503.19039v1},
}

@misc{BenedettiManivelPerrinGM,
  author = {Benedetti, Vladimiro and Manivel, Laurent and Perrin, Nicolas},
  title = {Quantum cohomology and irrationality of {Gushel--Mukai} fourfolds},
  year = {2026},
  eprint = {2603.17487v2},
  eprinttype = {arxiv},
  url = {https://arxiv.org/abs/2603.17487v2},
}

@misc{BenedettiEtAlAtomicCriterion,
  author = {Benedetti, Vladimiro and Fay, Aideen and Gu{\'e}r{\'e}, J{\'e}r{\'e}my and Manivel, Laurent and Perrin, Nicolas},
  title = {An atomic criterion for irrationality without quantum computations},
  year = {2026},
  eprint = {2607.26718v1},
  eprinttype = {arxiv},
  url = {https://arxiv.org/abs/2607.26718v1},
}

@misc{AlexeevEtAlFlagAtoms,
  author = {Alexeev, Boris and Cavenaghi, Leonardo F. and Galindo, Giovane and Georgiev, Bogdan and Katzarkov, Ludmil and Muniz Martins, Pedro Antonio},
  title = {On the atomic decomposition of complete intersection in flag varieties},
  year = {2026},
  eprint = {2609.21914v1},
  eprinttype = {arxiv},
  url = {https://arxiv.org/abs/2609.21914v1},
}

@misc{BenedettiEtAlVerra,
  author = {Benedetti, Vladimiro and Gu{\'e}r{\'e}, J{\'e}r{\'e}my and Manivel, Laurent and Perrin, Nicolas},
  title = {Quantum cohomology and birational geometry of {Verra} fourfolds},
  year = {2026},
  eprint = {2605.30450v1},
  eprinttype = {arxiv},
  url = {https://arxiv.org/abs/2605.30450v1},
}

@misc{CavenaghiKatzarkovKontsevichAtomsSymbols,
  author = {Cavenaghi, Leonardo F. and Katzarkov, Ludmil and Kontsevich, Maxim},
  title = {Atoms meet symbols},
  year = {2026},
  eprint = {2509.15831v4},
  eprinttype = {arxiv},
  url = {https://arxiv.org/abs/2509.15831v4},
}

@misc{FayEquivariantVerra,
  author = {Fay, Aideen},
  title = {Equivariant irrationality of very general symmetric {Verra} fourfolds},
  year = {2026},
  eprint = {2605.30439v1},
  eprinttype = {arxiv},
  url = {https://arxiv.org/abs/2605.30439v1},
}

@misc{FayHasseObstructions,
  author = {Fay, Aideen},
  title = {{Hasse} obstructions to rationality for special fourfolds},
  year = {2026},
  eprint = {2609.06759v1},
  eprinttype = {arxiv},
  url = {https://arxiv.org/abs/2609.06759v1},
}

@misc{BohningBothmerSuaNaiveAtoms,
  author = {B{\"o}hning, Christian and Graf von Bothmer, Hans-Christian and Su'a, Zac},
  title = {Naive atoms of blowups: examples},
  year = {2026},
  eprint = {2606.17884v2},
  eprinttype = {arxiv},
  url = {https://arxiv.org/abs/2606.17884v2},
}

@misc{GyengeProductQDM,
  author = {Gyenge, {\'A}d{\'a}m},
  title = {The quantum {$D$}-module of product varieties},
  year = {2025},
  eprint = {2509.07407v1},
  eprinttype = {arxiv},
  url = {https://arxiv.org/abs/2509.07407v1},
}

@article{TotaroHypersurfaces,
  author = {Totaro, Burt},
  title = {Hypersurfaces that are not stably rational},
  journaltitle = {Journal of the American Mathematical Society},
  volume = {29},
  number = {3},
  pages = {883--891},
  year = {2016},
  eprint = {1502.04040},
  eprinttype = {arxiv},
}

@article {DavidHertlingEulerFlow,
    AUTHOR = {David, Liana and Hertling, Claus},
     TITLE = {Regular {$F$}-manifolds: initial conditions and {F}robenius
              metrics},
   JOURNAL = {Ann. Sc. Norm. Super. Pisa Cl. Sci. (5)},
  FJOURNAL = {Annali della Scuola Normale Superiore di Pisa. Classe di
              Scienze. Serie V},
    VOLUME = {17},
      YEAR = {2017},
    NUMBER = {3},
     PAGES = {1121--1152},
      ISSN = {0391-173X,2036-2145},
   MRCLASS = {53D45 (32B10 32G99 53Z05)},
  MRNUMBER = {3726837},
}

@article{ChatzistamatiouLevineTorsion,
  author = {Chatzistamatiou, Andre and Levine, Marc},
  title = {Torsion orders of complete intersections},
  journaltitle = {Algebra \& Number Theory},
  volume = {11},
  number = {8},
  pages = {1779--1835},
  year = {2017},
  doi = {10.2140/ant.2017.11.1779},
  eprint = {1605.01913},
  eprinttype = {arxiv},
}

@article{NicaiseOttemTropical,
  author = {Nicaise, Johannes and Ottem, John Christian},
  title = {Tropical degenerations and stable rationality},
  journaltitle = {Duke Mathematical Journal},
  volume = {171},
  number = {15},
  pages = {3023--3075},
  year = {2022},
  doi = {10.1215/00127094-2022-0065},
  eprint = {1911.06138v4},
  eprinttype = {arxiv},
  url = {https://arxiv.org/abs/1911.06138v4},
}

@misc{SkauliQuadricCubic,
  author = {Skauli, Bj{\o}rn},
  title = {A stably irrational $(2,3)$-complete intersection fourfold over {$\mathbb Q$}},
  year = {2021},
  eprint = {2105.14846v1},
  eprinttype = {arxiv},
  url = {https://arxiv.org/abs/2105.14846v1},
}

@article{OkadaWeightedHypersurfaces,
  author = {Okada, Takuzo},
  title = {Smooth weighted hypersurfaces that are not stably rational},
  journaltitle = {Annales de l'Institut Fourier},
  volume = {71},
  number = {1},
  pages = {203--237},
  year = {2021},
  doi = {10.5802/aif.3390},
  eprint = {1709.07748v2},
  eprinttype = {arxiv},
  url = {https://arxiv.org/abs/1709.07748v2},
}

@article{HassettPirutkaTschinkelDoubleFourfold,
  author = {Hassett, Brendan and Pirutka, Alena and Tschinkel, Yuri},
  title = {A very general quartic double fourfold is not stably rational},
  journaltitle = {Algebraic Geometry},
  volume = {6},
  number = {1},
  pages = {64--75},
  year = {2019},
  doi = {10.14231/AG-2019-004},
  eprint = {1605.03220v1},
  eprinttype = {arxiv},
  url = {https://arxiv.org/abs/1605.03220v1},
}

@article{DeFernexRigidHypersurfaces,
  author = {de Fernex, Tommaso},
  title = {Birationally rigid hypersurfaces},
  journaltitle = {Inventiones Mathematicae},
  volume = {192},
  number = {3},
  pages = {533--566},
  year = {2013},
  eprint = {math/0604213},
  eprinttype = {arxiv},
}

@article{DeFernexRigidHypersurfacesErratum,
  author = {de Fernex, Tommaso},
  title = {Erratum to: Birationally rigid hypersurfaces},
  journaltitle = {Inventiones Mathematicae},
  volume = {203},
  number = {2},
  pages = {675--680},
  year = {2016},
  doi = {10.1007/s00222-015-0618-4},
  eprint = {1506.07086},
  eprinttype = {arxiv},
}

@incollection{TianGangQH,
  author = {Tian, Gang},
  title = {Quantum cohomology and its associativity},
  booktitle = {Current developments in mathematics, 1995 ({C}ambridge, {MA})},
  pages = {361--401},
  publisher = {International Press},
  location = {Cambridge, MA},
  year = {1996},
}

@misc{PomerleanoSeidel,
  author = {Daniel Pomerleano and Paul Seidel},
  title = {The quantum connection, {Fourier--Laplace} transform, and families of {$A_\infty$}-categories},
  year = {2026},
  eprint = {2308.13567v5},
  eprinttype = {arxiv},
  primaryclass = {math.SG},
  url = {https://arxiv.org/abs/2308.13567v5},
}

@article{HukuharaSingularII,
  author  = {Hukuhara, Masuo},
  title   = {Sur les points singuliers des {\'e}quations diff{\'e}rentielles lin{\'e}aires, {II}},
  journal = {Journal of the Faculty of Science, Hokkaido Imperial University. Series I},
  volume  = {5},
  pages   = {123--166},
  year    = {1937}
}

@article{TurrittinIrregular,
  author  = {Turrittin, H. L.},
  title   = {Convergent solutions of ordinary linear homogeneous differential equations
             in the neighborhood of an irregular singular point},
  journal = {Acta Mathematica},
  volume  = {93},
  pages   = {27--66},
  year    = {1955}
}

@article{LeveltJordan,
  author  = {Levelt, A. H. M.},
  title   = {Jordan decomposition for a class of singular differential operators},
  journal = {Arkiv f{\"o}r Matematik},
  volume  = {13},
  pages   = {1--27},
  year    = {1975},
  doi     = {10.1007/BF02386195}
}

@misc{CarazolliEtAlQuantumRationality,
  author        = {Carazolli, Luiz and
                   Cavenaghi, Leonardo F. and
                   Katzarkov, Ludmil and
                   Kontsevich, Maxim and
                   Muniz Martins, Pedro Antonio},
  title         = {A quantum connection approach to rationality},
  year          = {2026},
  eprint        = {2610.06669},
  archivePrefix = {arXiv},
  primaryClass  = {math.AG}
}

@misc{KuznetsovPerrySerre,
  author        = {Kuznetsov, Alexander and Perry, Alexander},
  title         = {Serre functors and dimensions of residual categories},
  year          = {2021},
  eprint        = {2109.02026},
  archivePrefix = {arXiv},
  primaryClass  = {math.AG}
}
\end{document}